\documentclass{amsart}
\usepackage{hyperref,amsmath,amsfonts,amssymb,amsthm,enumerate,mathtools,graphicx,xparse,xcolor,tikz,scalerel,stackengine}
\usepackage{tikz-3dplot}
\usepackage[a4paper, left=2cm, right=2cm, top=3cm, bottom=2cm]{geometry}
\usepackage[utf8]{inputenc}
\usepackage{enumitem}
\usepackage{float}
\usepackage{mdframed}

\graphicspath{{figures/}}
\usepackage{pgfplots}
\usepackage{subcaption} 
\usepackage{esint}   
\usepackage{xcolor} 
\usepackage{arydshln}   
\usepackage{comment}

\usepackage{accents}
\usepackage{empheq}
\usepackage{multimedia}
\usepackage{animate}
\usepackage[makeroom]{cancel}
\definecolor{morange}{rgb}{0.92, 0.51, 0.11}
\definecolor{mLightGreen}{HTML}{14B03D}
\definecolor{mPink}{HTML}{e1388d}
\definecolor{mRed}{HTML}{E83C3C}

\tikzstyle{dot}=[draw,fill=white,circle,inner sep=0pt,minimum size=4pt]

\pgfkeys{
  /pgf/declare function={Si1(\r,\nu,\d,\D)=2*\nu*\d*(\D/\d)^(\r-1);},
  /pgf/declare function={Si2(\r,\nu,\d,\D)=2*\nu*(1+\D^2/\d^2)^(\r/2-1)*\D*2^(1-\r/2);},
  /pgf/declare function={Si(\r,\q,\nuratio,\D)=Si1(\r,1/(1+\nuratio),1,\D)+Si2(\q,\nuratio/(1+\nuratio),1,\D);},
  /pgf/declare function={Sii1(\nu,\D)=2*\nu*\D;},
  /pgf/declare function={Dii(\rp,\nu,\d,\Stwo)=1/(2*\nu)*(1+\Stwo^2/(2*\nu*\d)^2)^(\rp/2-1)*\Stwo*2^(1-\rp/2);},
}

\numberwithin{dummy}{section}
\numberwithin{equation}{section}

\theoremstyle{plain}
\newtheorem{thm}{Theorem}[section]
\newtheorem{lemma}[thm]{Lemma}

\theoremstyle{plain}
\newtheorem{df}[thm]{Definition}

\theoremstyle{definition}
\newtheorem{rem}{Remark}
\newtheorem*{example}{Example}

\newcommand{\rr}{\mathbb{R}}

\newcommand{\nn}{\mathbb{N}}

\newcommand{\Cper}{C_\text{per}}
\newcommand{\Sper}{C^\infty_\text{per}}

\newcommand{\dx}{\,dx}
\newcommand{\dy}{\,dy}
\newcommand{\dt}{\,dt}
\newcommand{\ds}{\,ds}

\newcommand{\dert}{\partial_t}

\newcommand{\bn}{\pmb n}
\newcommand{\bxi}{\pmb \xi}

\newcommand{\bz}{\pmb z}

\newcommand{\tk}{\mathcal{T}_k}

\newcommand{\todo}[1] {\textcolor{red}{TODO: #1}}

\begin{document}

\title{Well--posedness and homogenization of a waste--protein clearance model in cortical tissue.}\thanks{LW and MB acknowledge support from project 90126 financed by GAUK and the project 25-16592S financed by GA\v{C}R. MER acknowledges support from Stiftelsen Kristian Gerhard Jebsen via the K.~G.~Jebsen Centre for Brain Fluid Research. AJ acknowledges partial support from the Austrian Science Fund (FWF), grant 10.55776/PAT2687825, and from the Austrian Federal Ministry for Women, Science and Research and implemented by \"OAD, project MULT09/2025. This work has received funding from the European Research Council (ERC) under the European Union's Horizon 2020 research and innovation programme, ERC Advanced Grant NEUROMORPH, no.~101018153. For open-access purposes, the authors have applied a CC BY public copyright license to any author-accepted manuscript version arising from this submission.}
\author[M.~Bul\'{i}\v{c}ek]{Miroslav Bul\'i\v{c}ek}
\address{Charles University, Faculty of Mathematics and Physics, Mathematical Institute\\
Sokolovsk\'{a} 83, 186 75 Prague, Czech Republic}
\email{mbul8060@karlin.mff.cuni.cz}

\author[A. J\"{u}ngel]{Ansgar J\"{u}ngel}
\address{TU Wien, Institute for Analysis and Scientific Computing \\
Wiedner Hauptstraße 8-10, 1040 Wien, Austria}
\email{juengel@tuwien.ac.at}

\author[M. E. Rognes]{Marie E. Rognes}
\address{Simula Research Laboratory, Department of Numerical Analysis and Scientific Computing \\ 
Kristian Augusts gate 23, 0164 Oslo, Norway \\ and
K.~G.~Jebsen Centre for Brain Fluid Research, University of Oslo, Oslo, Norway}
\email{meg@simula.no}

\author[L. Wintrov\'{a}]{Lucie Wintrov\'{a}}
\address{Charles University, Faculty of Mathematics and Physics, Mathematical Institute \\
Sokolovsk\'{a} 83, 186 75 Prague, Czech Republic}
\email{wintrova@karlin.mff.cuni.cz}

\begin{abstract}
We formulate and analyze a mathematical model for the clearance of waste proteins from cortical tissue. The model tracks three species of waste proteins, distinguished by their size, and incorporates constant production at neuronal surfaces, reaction–diffusion in the bulk of interstitial fluid, and nonlinear dynamics on capillary surfaces. We establish the existence, boundedness, and uniqueness of a solution to the microscopic model. 
Using two-scale convergence, we derive a homogenized model and prove that the microscopic solutions converge to the corresponding macroscopic limit. The resulting model reveals how the microscopic arrangement and relative volume fractions of neurons and capillaries affect waste-protein transport and clearance at the tissue scale. Numerical simulations illustrate these effects and demonstrate how variations in cortical microstructure translate into differences in the macroscopic clearance dynamics.
\end{abstract}
\keywords{Homogenization, two-scale convergence, reaction--diffusion systems, bulk--surface coupling, perforated domains, nonlinear reaction systems, multiscale analysis, receptor--ligand kinetics, amyloid-$\beta$ dynamics.}

\subjclass[2000]{35B27, 35K57, 35Q92, 92C37}

\maketitle
\section{Introduction}\label{sec:intro}
Alzheimer's disease is the most common neurodegenerative disorder, with its incidence increasing markedly with age \cite{Serot2012}. Although the mechanisms underlying the disease are complex and remain incompletely understood, the amyloid cascade hypothesis identifies the accumulation of amyloid-$\beta$ (A$\beta$) as an important hallmark of disease progression \cite{Goate1991, Hardy1992, Jack2010}. This hypothesis has been debated in recent years, as growing evidence indicates that tau pathology contributes substantially to neurodegeneration \cite{Cho2019, Singleton2025}. Nevertheless, A$\beta$ remains a central player in the disease, as it is thought to occur early in the pathological process and to promote tau pathology and neurotoxicity \cite{Potter2013, Hanseeuw2019}. A$\beta$ is continuously produced in the brain and occurs in several forms, ranging from soluble monomers and oligomers to larger aggregated species and plaques. Characterizing their dynamics and interactions is therefore essential for understanding A$\beta$ accumulation in brain tissue.

A$\beta$ dynamics arise from the interplay of production, transport, aggregation, degradation, and clearance \cite{tarasoff‐conway2015clearance}. A$\beta$ is predominantly produced by neurons and is released into the extracellular space (ECS), which is filled with interstitial fluid. Once released, A$\beta$ moves through the ECS either by diffusion or by bulk flow of the interstitial fluid~\cite{Cserr1981, Iliff2012, Bohr2022}. Here, we model transport in the ECS by diffusion alone, a modelling choice motivated by~\cite{Holter2017}. Within the ECS, A$\beta$  can undergo aggregation or degradation~\cite{CHEN2019, Elali2013}. In addition, soluble A$\beta$ is removed from the brain parenchyma by receptor-mediated transport across the blood--brain barrier (BBB)~\cite{hladky2018elimination, Pascale2011, Zlokovic2011}. Together, these processes couple the local A$\beta$ dynamics to the spatial organization of the surrounding tissue.

The dynamics of A$\beta$ take place across widely separated spatial scales. While cortical tissue extends over macroscopic distances and volumes of several cubic centimeters, A$\beta$ transport and clearance are governed by structures at the cellular scale, with neurons and capillaries having characteristic dimensions on the order of micrometers \cite{nicholson2011brain, SCHLAGETER1999}. The microscopic geometry determines the structure of the extracellular space and, consequently, the pathways available for diffusion as well as the extent and distribution of neuronal and capillary interfaces. Thus, the microscopic organization of the tissue directly affects both A$\beta$ transport and its production and clearance. Homogenization \cite{Allaire1992, Nguetseng} provides a natural framework for bridging these scales. By considering a periodically distributed microscopic structure and studying the limit as its characteristic size tends to zero, one obtains an effective macroscopic model whose coefficients incorporate the microscopic geometry and interface processes.

The main contribution of this work is the rigorous derivation of an effective macroscopic model for A$\beta$ clearance in cortical tissue from a microscopic description. Starting from a model that incorporates diffusion, aggregation, degradation, and receptor-mediated transport across the BBB, we derive a homogenized system whose effective coefficients explicitly reflect the underlying tissue microstructure. The resulting model provides a mathematically justified link between cellular-scale transport mechanisms and tissue-scale clearance dynamics, offering a quantitative framework for studying how microscopic tissue organization influences effective clearance.

The analysis faces a substantial difficulty arising from the strong coupling between the bulk aggregation network and the receptor-mediated clearance on the interfaces, which generates higher-order nonlinearities for which standard approaches do not directly yield the uniform estimates needed for homogenization. We show that, under natural structural assumptions on the reaction network, a modified Moser--De Giorgi iteration provides global $L^\infty$ bounds for all species, which ensure well-posedness of the microscopic model and allow us to pass to the macroscopic limit. The nonlinear interactions between the bulk species and surface-bound receptors on oscillating boundaries differ from the linear Robin-type bulk--surface couplings considered in \cite{AllaireDamlamianHornung, NeussRadu}, for which weak surface two-scale convergence is sufficient to identify the limiting boundary terms. Nonlinear bulk and boundary reactions in periodically structured media have also been studied in \cite{Ghan2016}. However, in our setting the nonlinear boundary flux is coupled to additional surface variables that evolve according to their own reaction kinetics. These receptor species have no surface diffusion, in contrast to the surface-diffusion models considered in \cite{Ptashnyk2020, Graf2014}, and therefore do not possess Sobolev spatial regularity. Instead, the analysis relies on uniform $L^\infty$ bounds for the receptor variables and on their coupling with the bulk concentrations. In contrast to existing nonlinear multiscale models of bulk--surface signalling \cite{Marciniak2008, Marciniak2012}, our model combines a multi-species bulk aggregation network with a multi-state surface receptor system. Moreover, rather than prescribing a specific reaction topology, both systems are described through general structural assumptions, including sign conditions, a conservation law for the receptor states, and an algebraic dissipativity condition on the aggregation network. 

The paper is organized as follows. Section~\ref{sec:micro} introduces the microscopic geometry and derives the mathematical model at the neuronal scale. Section~\ref{sec:wellposed} proves well-posedness of the micromodel and derives the necessary uniform estimates. Section~\ref{sec:conv} establishes the homogenization result, derives the corresponding macroscopic system, and analyzes it. Lastly, Section~\ref{sec:num} presents numerical simulations illustrating the influence of the microscopic structure on macroscopic diffusion of A$\beta$.

\section{Derivation of a micromodel for A$\beta$ clearance}
\label{sec:micro}
In this section, we present a derivation of a microscopic mathematical model for the clearance mechanisms in the cortical tissue. 
\subsection{Geometry representation} We consider a connected Lipschitz domain $\Omega \subset \mathbb{R}^3$ representing a segment of cortical tissue and assume a periodic distribution of cells within $\Omega$. This periodicity is understood as an idealization of the microscopic tissue structure, motivated by the separation between the cellular and macroscopic length scales. To describe the microscopic structure of the tissue, we follow the notation commonly used in the homogenization literature \cite{Marciniak2008,Ptashnyk2020,Neuss-Radu2007}.
We define a reference cell $Z = [0,l_1]\times[0,l_2]\times[0,l_3]$.
Within this cell, let $Y_N$ denote the union of neuronal regions, and $Y_C$ the union of capillary regions. The remaining connected subdomain, $Y = Z \setminus (Y_N \cup Y_C)$, represents the extracellular space. The resulting decomposition of the reference cell is illustrated in Figure~\ref{fig:geometry}. We denote the neuronal boundaries by $\Gamma_N := \partial Y_N$, and the capillary boundaries by $\Gamma_C := \partial Y_C$. The cell $Z$ is assumed to be periodically repeated across $\mathbb{R}^3$.

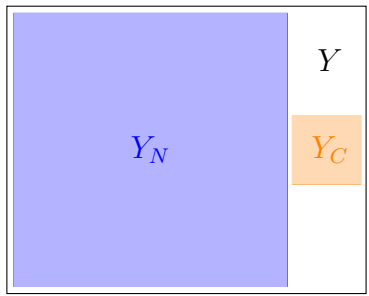
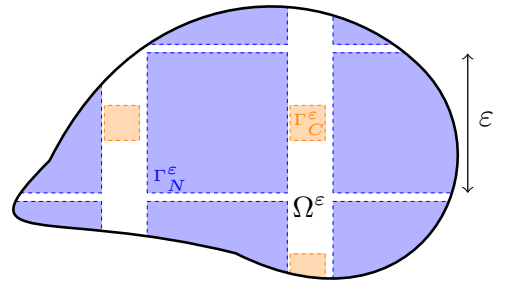
\begin{figure}[htbp]
    \centering
    \begin{subfigure}[b]{0.45\textwidth}
        \centering
        \begin{tikzpicture}[scale=0.95]
    \draw (0,0) rectangle (5,4);
    \node[black] at (4.5, 3.25) {\Large  $Y$};

    \draw[blue] (0.1, 0.1) rectangle (3.9, 3.9);
    \fill[blue!30] (0.1, 0.1) rectangle (3.9, 3.9);
    \node[blue] at (2, 2) {\Large $Y_N$};

    \draw[orange] (3.975, 1.525) rectangle (4.925, 2.475);
    \fill[orange!30] (3.975, 1.525) rectangle (4.925, 2.475);
    \node[orange] at (4.5, 2) {\Large $Y_C$};
\end{tikzpicture}
        \caption{The reference cell $Z$.} 
        \label{fig:homcell2D}
    \end{subfigure}
    \hfill
    \begin{subfigure}[b]{0.45\textwidth}
        \centering
        \includegraphics[width=\textwidth]{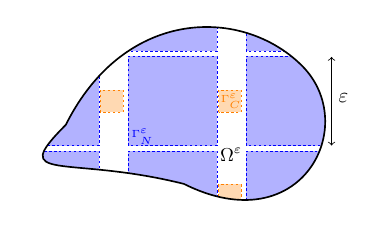}
        \vspace{-1.25cm}
        \caption{The perforated domain $\Omega^\varepsilon$.}
        \label{fig:domaingeometry}
    \end{subfigure}
    \caption{{\bf Constructing the microscopic domain $\Omega^\varepsilon$.} The reference cell $Z$ contains the neuronal region $Y_N$ and capillary region $Y_C$ within the extracellular space $Y$. The microscopic domain $\Omega^\varepsilon$ is obtained by periodically removing $\varepsilon$-scaled copies of $Y_N\cup Y_C$ from $\Omega$, with $\Gamma_N^\varepsilon$ and $\Gamma_C^\varepsilon$ denoting the corresponding interfaces.}
    \label{fig:geometry}
\end{figure}
 
For a fixed scale parameter $\varepsilon > 0$, the microscopic geometry within $\Omega$ is obtained by intersecting the $\varepsilon$-scaled periodic structure with $\Omega$. More precisely, we define \[\Xi := \left\{ \xi \in \mathbb{R}^3  :\; \xi = (l_1 z_1, l_2 z_2, l_3 z_3),\ \bz = (z_1,z_2,z_3) \in \mathbb{Z}^3,\ \varepsilon\big((Y_C \cup Y_N) + \xi\big) \subset \Omega \right\}\] and set
\begin{equation}\label{microdomain}
    \begin{aligned}
    \Omega_{CN}^\varepsilon := \bigcup_{\xi \in \Xi} \varepsilon (Y_C \cap Y_N+ \xi),   &&
    \Omega^\varepsilon := \Omega\setminus\Omega_{CN}, &&
    \Gamma_C^\varepsilon := \bigcup_{\xi \in \Xi} \varepsilon ( \Gamma_C + \xi), &&
    \Gamma_N^\varepsilon := \bigcup_{\xi \in \Xi} \varepsilon ( \Gamma_N + \xi).
\end{aligned}
\end{equation}
This construction is illustrated in Figure~\ref{fig:geometry} and the $\varepsilon$-scaling is illustrated in Figure~\ref{fig:homdom}.  For a fixed time horizon $T>0$, we further introduce the 
space-time cylinders
\[
Q := (0,T)\times\Omega, \qquad Q^\varepsilon := (0,T)\times\Omega^\varepsilon.
\]

\begin{figure}[h!]
    \centering
\includegraphics[width=\linewidth]{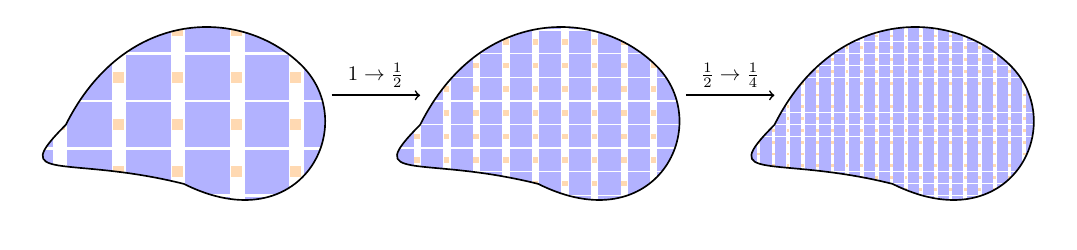}
\vspace{-3.5em}
\caption{{\bf Scaling of the microscopic geometry $\Omega^\varepsilon$.} The cortical microstructure from Figure \ref{fig:geometry} is periodically extended over the domain $\Omega$, scaled by $\varepsilon= 1, \frac{1}{2}, \frac{1}{4}$.}
    \label{fig:homdom}
\end{figure}

\subsection{Modelling A$\beta$ dynamics} A$\beta$ dynamics are modeled as a multi-species reaction--diffusion system with coupled bulk--surface interactions. Monomer production at neuronal interfaces induces a boundary source, while in the extracellular domain, the concentrations evolve by diffusion, degradation, and nonlinear aggregation, coupling the different species. 
Clearance is described by flux conditions at the capillary boundary, where bulk concentrations interact with surface-bound receptor densities via nonlinear binding and transport mechanisms.

The unknowns are the bulk concentrations $c_i^\varepsilon = c_i^\varepsilon(t,x)$ for $i=1,2,3$ and the surface densities $r_j^\varepsilon = r_j^\varepsilon(t,x)$ for $j=0,1,2,3$. Here, $c_1^\varepsilon$, $c_2^\varepsilon$, and $c_3^\varepsilon$ correspond to monomers, oligomers, and plaques, respectively, while $r_0^\varepsilon$ denotes free receptors and $r_j^\varepsilon$ receptors bound to the respective aggregation states. 
The coefficients $D_i > 0$ denote the bulk diffusion coefficients. 
The terms $f_{\mathrm{deg},i}(c_i^\varepsilon)$ and $f_{\mathrm{agg},i}(\vec{c}_\varepsilon)$ represent degradation and aggregation kinetics in the bulk. The boundary flux $q_{\mathrm{prod},i}$ models production at neuronal interfaces, while $f_{\mathrm{bind},i}(\vec{c}_\varepsilon,\vec{r}_\varepsilon)$ describes receptor-mediated exchange at the capillary boundary. The surface reactions $f_{\mathrm{rec},j}(\vec{c}_\varepsilon,\vec{r}_\varepsilon)$ encode the corresponding receptor dynamics.

The resulting system is given by 
\begin{align}
\partial_t c^\varepsilon_i 
- \nabla\!\cdot\!(D_i\nabla c^\varepsilon_i) 
&= f_{\mathrm{deg},i}(c^\varepsilon_i ) 
+ f_{\mathrm{agg},i}(\vec{c}_\varepsilon, \vec{c}_\varepsilon )
  &&\text{in } \Omega^\varepsilon , \label{ISFeps}
\\
 D_i\nabla c^\varepsilon_i \cdot \bn
&= \varepsilon q_{\mathrm{prod},i} &&\text{on } \Gamma^\varepsilon_N , \label{productioneps}
\\
D_i\nabla c^\varepsilon_i \cdot \bn
&= \varepsilon f_{\mathrm{bind},i}(\vec{c}_\varepsilon , \vec{r}_\varepsilon) &&\text{on } \Gamma^\varepsilon_C ,  \label{uptakeeps} \\
\partial_t r^\varepsilon_j  &= f_{\mathrm{rec},j}(\vec{c}_\varepsilon  , \vec{r}_\varepsilon ) &&\text{on } \Gamma^\varepsilon_C, \label{receptoreps}
\end{align}
for $i=1,2,3$ and $j=0,1,2,3$, where $\bn$ denotes the outer unit normal to $\Omega^\varepsilon$. 
The scaling with $\varepsilon$ in the boundary fluxes reflects the surface-to-volume ratio of the microstructure and ensures a nontrivial contribution of boundary processes in the homogenization limit.

The system is complemented by initial conditions
\begin{equation}\label{IC}
\begin{aligned}
c_i^\varepsilon(0,x) &= c_{i,0}^\varepsilon(x) &&\text{for } x\in \Omega^\varepsilon,\\
r_j^\varepsilon(0,x) &= r_{j,0}^\varepsilon(x) &&\text{for } x\in \Gamma_C^\varepsilon,
\end{aligned}
\end{equation}
and, for simplicity, by the no-flux condition on the outer boundary
\begin{align}
D_i\nabla c^\varepsilon_i \cdot \bn
&= 0 &&\text{on } \partial\Omega, \label{outereps}
\end{align}
which expresses that the domain is closed with respect to A$\beta$ transport at the macroscale.

\subsection{Notation} 
In what follows, we adopt the standard notation for Lebesgue, Sobolev, and Bochner spaces, equipping them with their usual norms. In addition, for a domain $\Omega \subset \mathbb{R}^d$, $d\in \nn $, $s \in (0,1)$ and $1 \leq p < \infty$, we define the fractional Sobolev space
\[
W^{s,p}(\Omega)
= \left\{ u \in L^p(\Omega) \;:\; [u]_{W^{s,p}(\Omega)} < \infty \right\},
\]
where the Gagliardo seminorm is given by
\[
[u]_{W^{s,p}(\Omega)}^p
= \int_\Omega \int_\Omega \frac{|u(x) - u(y)|^p}{|x - y|^{d + sp}} \, dx \, dy.
\]
The associated norm is defined as
\[
\|u\|_{W^{s,p}(\Omega)}^p
:= \|u\|_{L^p(\Omega)}^p + [u]_{W^{s,p}(\Omega)}^p.
\] Spatial vectors, such as $\mathbf{n} : \Omega \to \mathbb{R}^3$, are denoted by boldface symbols. In contrast, we write $\vec{c} = (c_1,c_2,c_3)$ for vectors representing mixtures or species concentrations. The distinction between these two types of quantities will always be clear from the context.

\section{Well-posedness and a priori estimates for the microscopic model}\label{sec:wellposed}
In this section, we prove the existence and uniqueness of a solution to the microscopic problem \eqref{ISFeps}--\eqref{outereps}. We also derive a priori estimates that allow us to pass to the limit as the number of homogenization cells tends to infinity.

The existence of a weak solution to \eqref{ISFeps}--\eqref{outereps} is established by means of a Schauder fixed-point argument~\cite{brezis2011}, which enables the decoupling of the surface reaction system from the bulk equations. The proof relies on the derivation of suitable a priori estimates for the bulk concentrations, exploiting the particular structure of the nonlinear reaction terms~\eqref{aggineq}. The $L^\infty$-bounds are established by means of an Alikakos-type iteration procedure~\cite{Alikakos1979}. The iteration relies on the quadratic growth structure of the bulk nonlinearities, combined with the higher integrability of the bulk concentrations, which allows the nonlinear terms to be controlled at each step and the iteration to be closed. Finally, uniqueness is shown by standard energy methods based on the established boundedness properties and Grönwall's inequality.

\subsection{Model assumptions}
We impose the following assumptions on the functions appearing in the model equations and on the initial data, which are motivated by biological considerations and further discussed in Remark~\ref{rem:as}.
\begin{enumerate}[label=A\arabic*.,ref=A\arabic*]
    \item \label{asagg}
    {\bf Bulk aggregation terms.} Let $\vec{a}, \vec{b} \in \mathbb{R}^3$. For $f_{\mathrm{agg},i}$, $i = 1,2,3$, we assume that
    \[
        f_{\mathrm{agg},i}(\vec{a}, \vec{b}) 
        = \sum_{j,l=1}^3 \kappa^i_{jl} a_l b_j 
        + \sum_{j=1}^3 \kappa^i_j b_j,
    \]
    where the coefficients satisfy $\kappa^i_{jl} \ge 0$ and $\kappa^i_j \ge 0$ for all $l$ and all $j \ne i$. 
    Furthermore, let
    \[S^i_{jl} := \kappa^i_{jl} + \kappa^i_{lj}, \qquad i,j,l = 1,2,3.\]
    We assume that there exist positive constants \(K_1,K_2,K_3\) such that
    \begin{equation}\label{aggineq}
        K_j S^j_{il} + K_i S^i_{jl} + K_l S^l_{ij} \le 0, \qquad \text{for } 1\leq i\leq j\leq l \leq 3.
    \end{equation}
    Since the left-hand side of \eqref{aggineq} is invariant under permutations of $i,j,l$ (as $S^p_{qr}=S^p_{rq}$), the inequality in fact holds for all $i,j,l = 1,2,3$.
       
    \item \label{asrec}
    {\bf Receptor dynamics terms.} Let $\vec{a} \in \mathbb{R}^3$ and $\vec{s} \in \mathbb{R}^4$. For $f_{\mathrm{rec},j}$, $j = 0,\dots,3$, we assume that
    \[
        f_{\mathrm{rec},j}(\vec{a}, \vec{s}) 
        = \sum_{i=0}^3 \sum_{l=1}^3 \lambda^j_{il} a_l s_i 
        + \sum_{i=0}^3 \lambda^j_i s_i,
    \]
    where the coefficients satisfy $\lambda^j_{il} \ge 0$ and $\lambda^j_i \ge 0$ for all $i \ne j$ and all $l$. 
    Moreover, we assume
    \begin{align}
        \sum_{j=0}^3 f_{\mathrm{rec},j}(\vec{a}, \vec{s}) = 0. \label{receptorconserved}
    \end{align}

\item \label{asbind}
    {\bf Surface binding terms.} Let $\vec{a} \in \mathbb{R}^3$ and $\vec{s} \in \mathbb{R}^4$. For $f_{\mathrm{bind},j}$, $i = 1,2,3$, we assume that
    \[
        f_{\mathrm{bind},i}(\vec{a}, \vec{s}) 
        = \sum_{j=0}^3 \beta^i_j a_i s_j 
        + \sum_{j=0}^3 \gamma^i_j s_j,
    \]
    where the coefficients satisfy $\gamma^i_j \ge 0$ and $\beta^i_j \le 0$ for all $j$.

    \item \label{asdeg}
     {\bf Degradation terms.} Let $\vec{a} \in \mathbb{R}^3$. For $i=1,2,3$, the degradation terms are given by
    \[
        f_{\mathrm{deg},i}(a_i) = -\delta_i a_i,
    \]
    where $\delta_i >0 $.

    \item \label{asprod}
    {\bf Production terms.} For $i=1,2,3$, we assume that the production terms satisfy  $q_{\mathrm{prod},i} \geq 0$ a.e. on $ \Gamma_C^\varepsilon\times (0,T)$. Furthermore we assume \[\|q_{\mathrm{prod},i}\|_{L^\infty(\Gamma_N^\varepsilon\times (0,T))} \leq C.\]

    \item \label{asdif}
    {\bf Bulk diffusion coefficients.} For $i=1,2,3$, the diffusivity coefficients $D_i$ are strictly positive and constant.

    \item \label{asinit}
    {\bf Initial data.} For all $i$ and $j$, we assume that the initial data satisfy $c_{i,0}^\varepsilon \ge 0$ a.e. in $\Omega^\varepsilon$ and $r_{j,0}^\varepsilon \ge 0$ a.e. on $\Gamma_C^\varepsilon.$ Furthermore, we assume
    \[\|c_{i,0}^\varepsilon\|_{L^\infty(\Omega^\varepsilon)}\leq C, \qquad \|r_{j,0}^\varepsilon\|_{L^\infty(\Gamma_C^\varepsilon)} \leq C,\] where $C$  is independent of $\varepsilon$. 
\end{enumerate}

\begin{rem}[Discussion of the assumptions]\label{rem:as}
Assumption \ref{asagg} describes aggregation through pairwise interactions between molecular species, resulting in reaction terms that are at most quadratic in the concentrations. The nonnegativity conditions on the coefficients ensure that the corresponding interactions do not introduce negative production terms. Condition \eqref{aggineq} imposes a weighted dissipativity condition on the quadratic terms, which is used to obtain uniform a priori estimates for the total concentration. In particular, if \(i,j,l\) are pairwise distinct, then every term in \eqref{aggineq} is nonnegative by the sign conditions imposed on  \(\kappa^i_{jl}\). Consequently, \eqref{aggineq} can hold only if all such coefficients vanish. Thus, the model excludes quadratic reactions involving three distinct species, such as interactions of the form $c_1+c_2\to c_3$.

Assumption \ref{asrec} models receptor dynamics through mass-action kinetics on the cell membrane. The conservation law \eqref{receptorconserved} expresses the fact that receptors may change state through binding and unbinding events but are neither created nor destroyed.
The binding terms in Assumption \ref{asbind} are consistent with mass-action kinetics. The sign conditions encode that binding decreases the concentration of free ligands, whereas dissociation processes contribute positively to the concentration of unbound species.
The degradation terms in Assumption \ref{asdeg} represent natural clearance mechanisms and are assumed to be proportional to the local concentration. Assumption \ref{asprod} reflects the fact that production rates are nonnegative, while Assumption \ref{asdif} corresponds to standard diffusive transport with positive diffusivities.
Finally, Assumption \ref{asinit} requires biologically meaningful nonnegative initial concentrations and receptor densities. The uniform $L^\infty$-bounds exclude unrealistically large initial values and provide estimates that are independent of the microscopic scale $\varepsilon$.
\end{rem}

\begin{example}[Model satisfying the assumptions]
The following constitutive relations, based on a three-species aggregation model introduced in \cite{Heppner2015, CHEN2019}, provide an example satisfying Assumptions \ref{asagg}–\ref{asinit}. The bulk aggregation terms are given by
\begin{equation*}
    \begin{aligned}
 f_{\mathrm{agg},1} ( \vec{c}, \vec{c})&= - 2 k_{11} c_1^2 - k_{12} c_1 c_2 - k_{13} c_1 c_3 
                   + (2 \gamma_{20} + \gamma_{21}) c_2 + \gamma_{31} c_3, \\
f_{\mathrm{agg},2} ( \vec{c}, \vec{c}) &= k_{11} c_1^2 
                   -  (1+\alpha) k_{22} c_2^2 - k_{23} c_2 c_3
                   + (\gamma_{22} - \gamma_{20}) c_2 
                   + ( 2\gamma_{30} + \gamma_{32} ) c_3, \\
f_{\mathrm{agg},3} ( \vec{c}, \vec{c}) &=  \alpha k_{22} c_2^2 
                   -  \gamma_{30} c_3.
\end{aligned}
\end{equation*}
  with receptor dynamics
  \begin{equation*}
  \begin{aligned}
  f_{\mathrm{rec},0} ( \vec{c}, \vec{r}) = - \sum_{i=1}^3 k_{\mathrm{on},i} c_i r_0 + \sum_{i=1}^3 k_{\mathrm{off},i} r_i, \qquad
  f_{\mathrm{rec},i} ( \vec{c}, \vec{r}) = k_{\mathrm{on},i} c_i r_0 - k_{\mathrm{off},i} ( \vec{c}, \vec{r}) r_i,\quad i=1,2,3,
  \end{aligned}
  \end{equation*}
   and binding terms
  \begin{equation*}
  \begin{aligned}
  f_{\mathrm{bind},1}( \vec{c}, \vec{r}) = - k_{\mathrm{on},1} c_1 r_0,\qquad
  f_{\mathrm{bind},2}( \vec{c}, \vec{r}) = - k_{\mathrm{on},2} c_2 r_0 + \tfrac{1}{2} k_{\mathrm{off},2} r_2,\qquad
  f_{\mathrm{bind},3}( \vec{c}, \vec{r}) = - k_{\mathrm{on},3} c_3 r_0 + k_{\mathrm{off},3} r_3,
  \end{aligned}
  \end{equation*}
where all coefficients are positive. One directly verifies that Assumptions \ref{asrec} and \ref{asbind} are satisfied. Moreover, taking $K_1=1$, $K_2=\frac{k_{12}}{k_{k11}}$, and $K_3=\frac{k_{12}k_{23}}{\alpha k_{k11}k_{22}}$ in \eqref{aggineq}, Assumption \ref{asagg} is also satisfied. If the initial data are bounded and nonnegative, the remaining conditions follow directly from the structure of the diffusion, production, and degradation terms.
  \end{example}

\subsection{Weak solution to the microscopic model}
We now introduce the notion of a weak solution to the microscopic problem.
\begin{df}[Weak solution]\label{weaksoldef}
Let $T>0$ be fixed. A pair of functions 
\[
(\vec{c}_\varepsilon, \vec{r}_\varepsilon) = \bigl( c_1^\varepsilon,c_2^\varepsilon,c_3^\varepsilon, r_0^\varepsilon, r_1^\varepsilon, r_2^\varepsilon, r_3^\varepsilon \bigr)
\]
is called a weak solution to \eqref{ISFeps}--\eqref{outereps} on $(0,T)$ if, for all $i=1,2,3$ and $j=0,1,2,3$,
\begin{equation}\label{wfspace}
    \begin{gathered}
    c_i^\varepsilon \in L^2(0,T; W^{1,2}(\Omega^\varepsilon)), 
\quad \partial_t c_i^\varepsilon \in L^2(0,T; W^{-1,2}(\Omega^\varepsilon)), \\
r_j^\varepsilon \in 
L^\infty((0,T)\times \Gamma_C^\varepsilon), 
\quad \partial_t r_j^\varepsilon \in L^2(0,T; L^2(\Gamma_C^\varepsilon)),
\end{gathered}
\end{equation}
and for all test functions $\varphi_i \in L^2(0,T; W^{1,2}(\Omega^\varepsilon))$ and $\psi_j \in L^2(0,T; L^2(\Gamma_C^\varepsilon))$, it holds
\begin{align}\label{wfbulk}
\int_0^T \!\!&\left\langle \partial_t c_i^\varepsilon, \varphi_i \right\rangle \, dt 
 + \int_0^T \!\!\int_{\Omega^\varepsilon} D_i \nabla c_i^\varepsilon \cdot \nabla \varphi_i \, dx\,dt
= \varepsilon \int_0^T \!\!\int_{\Gamma_N^\varepsilon} q_{\mathrm{prod},i} \, \varphi_i \, dS\,dt  \\& + \varepsilon \int_0^T \!\!\int_{\Gamma_C^\varepsilon} f_{\mathrm{bind},i}(\vec{c}_\varepsilon,\vec{r}_\varepsilon) \, \varphi_i \, dS\,dt + \int_0^T \!\!\int_{\Omega^\varepsilon} \bigl( f_{\mathrm{deg},i}(c_i^\varepsilon)
+ f_{\mathrm{agg},i}(\vec{c}_\varepsilon, \vec{c}_\varepsilon) \bigr) \varphi_i \, dx\,dt
\nonumber,
\end{align}
and
\begin{align}\label{wfsurf}
\int_0^T \!\!\int_{\Gamma_C^\varepsilon} \partial_t r_j^\varepsilon \, \psi_j \, dS\,dt
= \int_0^T \!\!\int_{\Gamma_C^\varepsilon} f_{\mathrm{rec},j}(\vec{c}_\varepsilon,\vec{r}_\varepsilon) \, \psi_j \, dS\,dt.
\end{align}
Moreover, the initial conditions \eqref{IC} are attained in the sense that
\[
c_i^\varepsilon(t) \to c_{i,0}^\varepsilon
\text{ in } L^2(\Omega^\varepsilon)
\quad \text{and} \quad
r_j^\varepsilon(t) \to r_{j,0}^\varepsilon
\text{ in } L^2(\Gamma_C^\varepsilon)
\quad \text{as } t\to0^+.
\]
\end{df}
In the subsequent analysis, we shall repeatedly use the following scaled trace inequalities. The classical $W^{1,2}$--inequality will be used in the derivation of the a priori estimates, while the fractional version is included for later use in proving the strong convergence of the traces.
\begin{rem}[Scaled trace inequality]\label{rem:scaletrace}
The following scaled trace inequality holds
\begin{equation}\label{scaledtrace}
\varepsilon \|c\|_{L^2(\Gamma^\varepsilon)}^2
\le
C_\delta \|c\|_{L^2(\Omega^\varepsilon)}^2
+
\varepsilon^2 \delta \|\nabla c\|_{L^2(\Omega^\varepsilon)}^2,
\end{equation}
for every \(c\in W^{1,2}(\Omega^\varepsilon)\). The constant
\(C_\delta\) is independent of \(\varepsilon\) and satisfies
\[
C_\delta\le
C\left(1+\frac1\delta\right),
\]
where \(C>0\) depends only on the reference geometry (in particular on
\(Y\) and \(\Gamma\)). The factor \(\varepsilon\) in front of the
boundary norm reflects the natural scaling of oscillating surfaces, compensating for the difference in dimension between the volume and surface measures. See \cite[Lemma 3]{Hornung1991} for the proof.

Furthermore, for every \(c\in W^{r,2}(\Omega^\varepsilon)\) with
\(1/2<r<1\), the fractional trace inequality
\begin{equation}\label{scaledfractrace}
    \varepsilon \|c\|_{L^2(\Gamma^\varepsilon)}^2
\le
C\|c\|_{L^2(\Omega^\varepsilon)}^2
+
C\varepsilon^{2r}
[c]_{W^{r,2}(\Omega^\varepsilon)}^2,
\end{equation}
holds, where \(C>0\) may depend on \(r\), but is independent of \(\varepsilon\). For the proof see \cite[Lemma 4.2]{Marciniak2008}. 
\end{rem}

In order to derive uniform estimates and analyze the limit as $\varepsilon \to 0$ on the fixed macroscopic domain, we employ an extension operator from $\Omega^\varepsilon$ to $\Omega$, see \cite{Acerbi1992, Cioranescu1979, Hornung1991}.
\begin{rem}[Extension operator]\label{rem:extension}
The geometric assumptions imposed on $\Omega^\varepsilon$, most notably its connectedness, guarantee the existence of an extension operator such that for every $c \in W^{1,2}(\Omega^\varepsilon)$ there exists an extension $\tilde{c}\in W^{1,2}(\Omega)$ satisfying $\tilde{c}=c$ a.e. in  $\Omega^\varepsilon$ and the uniform estimates
\begin{equation}\label{ext_est}
\|\tilde{c}\|_{L^2(\Omega)} \le C \|c\|_{L^2(\Omega^\varepsilon)} 
\quad \text{and} \quad 
\|\nabla \tilde{c}\|_{L^2(\Omega)} \le C \|\nabla c\|_{L^2(\Omega^\varepsilon)},
\end{equation}
where the constant $C > 0$ is independent of $\varepsilon$. For time-dependent functions, this extension is defined pointwise in time, and the linearity of the operator ensures that $\tilde{c} \in L^2(0,T; W^{1,2}(\Omega))$.
\end{rem}

We now state the main result of this section, establishing the existence and uniqueness of a weak solution to \eqref{ISFeps}--\eqref{outereps}, which satisfies a priori estimates that are uniform with respect to \(\varepsilon\).
\begin{thm}\label{wellposednes}
    For every fixed $\varepsilon > 0$, under Assumptions \ref{asagg}--\ref{asinit}, there exists a unique nonnegative weak solution $(\vec{c}_\varepsilon, \vec{r}_\varepsilon)$ to the problem \eqref{ISFeps}--\eqref{outereps}. Moreover, this solution satisfies the following estimates:
\begin{align}
        \|c_i^\varepsilon\|_{L^\infty(0,T; L^2(\Omega^\varepsilon))} + \|\nabla c_i^\varepsilon\|_{L^2((0,T)\times\Omega^\varepsilon)} +  \sqrt{\varepsilon}
        \|c_i^\varepsilon\|_{L^2((0,T) \times \Gamma^\varepsilon_C)} \leq C, \label{estimatec} \\
     \|r_j^\varepsilon\|_{L^\infty((0,T)\times \Gamma_C^\varepsilon)} \leq C, \label{estimater}\\
     \|\partial_t c_i^\varepsilon\|_{L^2(0,T; W^{-1,2}(\Omega^\varepsilon))} +  \|\partial_t r_j^\varepsilon\|_{L^2(0,T; L^2(\Gamma_C^\varepsilon))} \leq C, \label{estimateder}\\
     \|c_i^\varepsilon\|_{L^\infty((0,T)\times \Gamma_C^\varepsilon)} +  \|c_i^\varepsilon\|_{L^\infty((0,T)\times\Omega^\varepsilon)} \leq C, \label{estimatecbetter}
\end{align}
    for $i=1,2,3$, $j=0,1,2,3$, where the constant $C$ is independent of $\varepsilon$. 
\end{thm}

We outline the proof of Theorem~\ref{wellposednes} by splitting it into several steps. Each of the following lemmata holds for fixed $\varepsilon > 0$. In Lemma~\ref{existencews}, existence of a nonnegative weak solution is obtained via a Schauder's fixed-point argument applied to a suitably truncated system. Lemma~\ref{boundedws} provides $L^\infty$-bounds by an Alikakos iteration. Finally, Lemma~\ref{uniquews} establishes uniqueness of weak solutions by a standard energy estimate and Gronwall’s lemma.
\begin{lemma}\label{existencews}
    Under Assumptions \ref{asagg}--\ref{asinit}, there exists a nonnegative weak solution to the microscopic problem \eqref{ISFeps}--\eqref{receptoreps}. This solution satisfies the uniform estimates \eqref{estimatec}--\eqref{estimateder}.
\end{lemma}
\begin{proof}
The proof relies on a standard approximation procedure for nonlinear systems. To handle the nonlinear coupling in the bulk and the surface, we first introduce a truncated problem.
More precisely, for any $k>0$, we define the cut-off function
\begin{align}\label{tk}
    &\tk^+:\rr \to [0,k], &\tk^+(z) := \max\{0,\min\{z,k\}\}.
\end{align}
We then modify the receptor equation \eqref{receptoreps} by applying the cut-off to the receptor concentrations:
\begin{align}
\partial_t r^\varepsilon_j
= f_{\mathrm{rec},j}(\vec{c}_\varepsilon, \tk^+(\vec{r}_\varepsilon))
\qquad \text{on } \Gamma^\varepsilon_C, \label{cutreceptoreps}
\end{align}
where the truncation is understood component-wise. 
We further modify the bulk equations \eqref{ISFeps} by truncating the negative part of bulk concentrations $c_i^+ := \max\{0,c_i\}$ in the nonlinear aggregation term. We consider 
\begin{align}
    \partial_t c^\varepsilon_i 
            - \nabla \cdot (D_i \nabla c^\varepsilon_i) 
            = f_{\mathrm{deg},i}(c^\varepsilon_i) 
            +  f_{\mathrm{agg},i}(\vec{c}_{\varepsilon,+}, \vec{c}_{\varepsilon,+})
            \qquad \text{in } \Omega^\varepsilon \label{cutISFeps},
\end{align}
where the positive part is understood component-wise.

Let $r \in \left(\tfrac{1}{2},1\right)$ be fixed and define
    \begin{align}\label{seta}
        \mathcal{A} := \Big\{ \vec{d}_\varepsilon \in \left(L^2\big(0,T; W^{r,2}(\Omega^\varepsilon)\big)\right)^3 : \; \|\vec{d}_\varepsilon\|_{L^2(W^{r,2}(\Omega^\varepsilon))} \leq C_{ \mathcal{A} }, \;
        d_i^\varepsilon \ge 0 \ \text{a.e. in } (0,T)\times\Omega^\varepsilon,\; i=1,2,3 \Big\},
    \end{align}
    where the constant $C_{ \mathcal{A} }$ will be specified later.
    To prove the existence of a solution to the system \eqref{cutISFeps}, \eqref{productioneps}--\eqref{uptakeeps}, \eqref{cutreceptoreps}, we employ a fixed-point argument. 
   More precisely, for a given $\vec{d}_\varepsilon \in \mathcal{A}$, we introduce the operator $F : \mathcal{A} \to \mathcal{A}$ by setting
    $ F(\vec{d}_\varepsilon):= \vec{c}_\varepsilon $, where $\vec{c}_\varepsilon$ solves the system
    \begin{equation}\label{systemfixed}
        \begin{aligned}
            \partial_t c^\varepsilon_i 
            - \nabla \cdot (D_i \nabla c^\varepsilon_i) 
            &= f_{\mathrm{deg},i}(c^\varepsilon_i) 
            +  f_{\mathrm{agg},i}(\vec{c}_{\varepsilon,+}, \vec{c}_{\varepsilon,+})
            && \text{in } \Omega^\varepsilon, \\
            D_i \nabla c^\varepsilon_i \cdot\bn
            &= \varepsilon q_{\mathrm{prod},i}
            && \text{on } \Gamma^\varepsilon_N, \\
            D_i \nabla c^\varepsilon_i \cdot\bn
            &= \varepsilon f_{\mathrm{bind},i}(\vec{c}_\varepsilon, \vec{r}_\varepsilon)
            && \text{on } \Gamma^\varepsilon_C, \\
            \partial_t r^\varepsilon_j
            &= f_{\mathrm{rec},j}\big(\vec{d}_\varepsilon, \tk^+(\vec{r}_\varepsilon)\big)
            && \text{on } \Gamma^\varepsilon_C,
        \end{aligned}
    \end{equation}
    for $i=1,2,3$ and $j=0,1,2,3$, supplemented with the initial conditions \eqref{IC}. In this formulation, the receptor evolution no longer depends on $\vec{c}_\varepsilon$, which allows us to solve for $\vec{r}_\varepsilon$ independently. 

For a given $\vec{d}_\varepsilon \in \mathcal{A}$, we seek $\vec{r}_\varepsilon$ satisfying
\begin{equation}\label{wftkrec}
\int_{\Gamma_C^\varepsilon} \partial_t r^\varepsilon_j  \psi_j \ds
= \int_{\Gamma_C^\varepsilon} f_{\mathrm{rec},j}\big(\vec{d}_\varepsilon, \tk^+(\vec{r}_\varepsilon)\big) \psi_j \ds,
\quad \forall \psi_j \in L^2(0,T; L^2(\Gamma_C^\varepsilon)), \; j=0,1,2,3.
\end{equation}
We formally test \eqref{wftkrec} by $\varepsilon r_j^\varepsilon$, sum over $j$, and invoke Assumption~\ref{asrec}. Combining Hölder’s inequality, Young’s inequality, and the scaled trace inequality  \eqref{scaledtrace} yields the following a priori estimate
\begin{equation}\label{aprioritkrec}
     \frac{d}{dt} \frac{1}{2} \varepsilon \|\vec{r}_\varepsilon\|_{L^2(\Gamma_C^\varepsilon)}^2\leq \varepsilon k \int_{\Gamma_C^\varepsilon} \left(1+|\vec{d}_\varepsilon|\right) |\vec{r}_\varepsilon| \, dS \leq C(k) \varepsilon \|\vec{r}_\varepsilon\|_{L^2(\Gamma_C^\varepsilon)}^2 + C(k)(1+\|\vec{d}_\varepsilon\|_{W^{r,2}(\Omega^\varepsilon)}^2).
\end{equation}
An application of Gronwall’s lemma yields $\sqrt{\varepsilon}\|\vec{r}_\varepsilon\|_{L^\infty(L^2(\Gamma_C^\varepsilon))} \leq C(k) $, where the constant $C$ is independent of $\varepsilon$, but depends on $k$ and blows up as $k \to \infty$.  
Next, formally testing \eqref{wftkrec} with $\psi_j = \partial_t r_j^\varepsilon$, summing over $j$, and using the previous bound, we obtain $\|\dert \vec{r}_\varepsilon\|_{L^2(L^2(\Gamma_C^\varepsilon))} \leq C(k)$. The above a priori estimates allow one to rigorously establish the existence of a solution $\vec{r}_\varepsilon$ to \eqref{wftkrec} via a standard Galerkin approximation argument, which we omit for brevity.

To show nonnegativity, we test \eqref{wftkrec} with $r_j^{\varepsilon,-}:=\min\{0,r_j\}$. We have 
\begin{equation}\label{nonnegativerk}
    \frac{d}{dt} \|r_j^{\varepsilon,-}\|_{L^2(\Gamma_C^\varepsilon)}^2 = \sum_{i=0}^3 \int_{\Gamma_C^\varepsilon}\left(\sum_{l=1}^3\lambda^j_{il}d^\varepsilon_l + \lambda^j_i \right) \tk^+(r_i^\varepsilon) r_j^{\varepsilon, -} \, dS \leq  \int_{\Gamma_C^\varepsilon}\left(\sum_{l=1}^3\lambda^j_{jl}d^\varepsilon_l + \lambda^j_j \right) \tk^+(r_j^\varepsilon) r_j^{\varepsilon, -} \, dS = 0,
\end{equation}
where the inequality follows from the nonnegativity of $\lambda_{il}^j$ and $\lambda_i^j$ for $i\ne j$, stated in Assumption~\ref{asrec}, together with the nonnegativity of the components of $\vec{d}_\varepsilon$.
Exploiting the nonnegativity of the initial data given in Assumption~\ref{asinit} together with Gronwall’s lemma, we deduce \[\|\vec{r}_{\varepsilon,-}\|_{L^\infty(L^2(\Gamma_C^\varepsilon))}^2\leq 0,\] which implies $r_j^\varepsilon \geq 0$ almost everywhere on $\Gamma_N^\varepsilon \times (0,T)$ for all $j$. Finally, by Assumption~\ref{asrec}, the total receptor density is conserved, i.e.,  \(\dert \sum_{j=0}^3 r_j=0.\) Combined with the nonnegativity of each component and the boundedness of initial data given by Assumption~\ref{asinit}, this yields the uniform bound 
\begin{equation}
    \|\vec{r}_\varepsilon\|_{L^\infty((0,T)\times \Gamma_C^\varepsilon)} \leq C \|\vec{r}_{\varepsilon,0}\|_{L^\infty(\Gamma_C^\varepsilon)} \leq C,
\end{equation} with $C$ independent of $\varepsilon$ or the choice of $\vec{d}_\varepsilon$. 

Next, for the previously constructed $\vec{r}_\varepsilon$, we seek $\vec{c}_\varepsilon$ solving
\begin{align}\label{wftkbulk}
\int_0^T \!\!&\left\langle\partial_t c_i^\varepsilon, \varphi_i \right\rangle \, dt 
 + \int_0^T \!\!\int_{\Omega^\varepsilon} D_i \nabla c_i^\varepsilon \cdot \nabla \varphi_i \, dx\,dt
= \varepsilon \int_0^T \!\!\int_{\Gamma_N^\varepsilon} q_{\mathrm{prod},i} \, \varphi_i \, dS\,dt \\&
 + \varepsilon \int_0^T \!\!\int_{\Gamma_C^\varepsilon} f_{\mathrm{bind},i}(\vec{c}_\varepsilon,\vec{r}_\varepsilon) \, \varphi_i \, dS\,dt +\int_0^T \!\!\int_{\Omega^\varepsilon} \bigl( f_{\mathrm{deg},i}(c_i^\varepsilon)
+ f_{\mathrm{agg},i}(\vec{c}_{\varepsilon,+}, \vec{c}_{\varepsilon,+}) \bigr) \varphi_i \, dx\,dt
\nonumber,
\end{align}
for $\varphi_i \in L^2(0,T; W^{1,2}(\Omega^\varepsilon))$ and $i=1,2,3$.
To show the existence of such $\vec{c}_\varepsilon$, we start by formally testing \eqref{wftkbulk} with $\varphi_i = c_i^\varepsilon$. Using Assumption~\ref{asdeg}, we obtain the following identity
\begin{equation}\label{cexist}
    \begin{aligned}
        \frac{1}{2}\frac{d}{dt}\|c_i^\varepsilon\|_{L^2(\Omega^\varepsilon)}^2 + D_i\|\nabla c_i^\varepsilon&\|_{L^2(\Omega^\varepsilon)}^2 +  \delta_i\|c_i^\varepsilon\|_{ L^2(\Omega^\varepsilon)}^2\\
&= \varepsilon \int_{\Gamma_N^\varepsilon} q_{\mathrm{prod},i} c_i^\varepsilon \, dS 
 + \varepsilon \int_{\Gamma_C^\varepsilon} f_{\mathrm{bind},i}(\vec{c}_\varepsilon,\vec{r}_\varepsilon) c_i^\varepsilon \, dS +\int_{\Omega^\varepsilon}
 f_{\mathrm{agg},i}(\vec{c}_{\varepsilon,+}, \vec{c}_{\varepsilon,+})  c_i^\varepsilon \, dx.
    \end{aligned}
\end{equation}
By Assumption~\ref{asbind} and the nonnegativity of $r^\varepsilon_j$, we have
\[\int_{\Gamma_C^\varepsilon} f_{\mathrm{bind},i}(\vec{c}_\varepsilon,\vec{r}_\varepsilon) c_i^\varepsilon \, dS \leq  \int_{\Gamma_C^\varepsilon} \sum_{j=0}^3 \gamma^i_j |r^\varepsilon_j| |c_i^\varepsilon| \, dS. \] 

We now make use of the inequality \eqref{aggineq}, formulated in Assumption~\ref{asagg}. We multiply \eqref{cexist} by $K_i$ and sum over $i$. For the aggregation term, we have
\begin{align*}
   &\sum_{i,j=1}^3 \int_{\Omega^\varepsilon} K_i \left( \sum_{l=1}^3   \kappa^i_{jl} c_l^{\varepsilon,+}
        + \kappa^i_j\right) c_j^{\varepsilon,+}c_i^{\varepsilon} \dx  \leq \sum_{i,j=1}^3 \int_{\Omega^\varepsilon} K_i \left( \sum_{l=1}^3   \kappa^i_{jl} c_l^{\varepsilon,+}
        + \kappa^i_j\right) c_j^{\varepsilon,+}c_i^{\varepsilon,+} \dx\\
        & = \frac{1}{6}\sum_{i,j,l=1}^3 \int_{\Omega^\varepsilon} (K_j S^j_{il} + K_i S^i_{jl} + K_l S^l_{ij})  c_l^{\varepsilon,+}c_j^{\varepsilon,+} c_i^{\varepsilon,+} \dx + \sum_{i,j=1}^3 \int_{\Omega^\varepsilon} K_i
        \kappa^j_i c_j^{\varepsilon,+}c_i^{\varepsilon,+} \dx  \\ &\leq \sum_{i,j=1}^3 \int_{\Omega^\varepsilon} K_i
        \kappa^j_i c_j^{\varepsilon,+}c_i^{\varepsilon,+} \dx\leq \sum_{i,j=1}^3 \int_{\Omega^\varepsilon} K_i
        \kappa^j_i |c_j^{\varepsilon}||c_i^{\varepsilon}|. 
\end{align*}
 We thus arrive at 
\begin{multline*}
\sum_{i=1}^3K_i\left(\frac{1}{2}\frac{d}{dt}\|c_i^\varepsilon\|_{L^2(\Omega^\varepsilon)}^2 + D_i\|\nabla c_i^\varepsilon\|_{L^2(\Omega^\varepsilon)}^2 +  \delta_i\|c_i^\varepsilon\|_{ L^2(\Omega^\varepsilon)}^2\right) \\
\leq \sum_{i=1}^3\varepsilon K_i  \int_{\Gamma_N^\varepsilon} |q_{\mathrm{prod},i}|| c_i^\varepsilon| \, dS 
 + \sum_{i=1}^3\!\sum_{j=0}^3\varepsilon K_i \int_{\Gamma_C^\varepsilon}  |r^\varepsilon_j| |c_i^\varepsilon| \, dS +\sum_{i,j=1}^3 K_i
        \kappa^j_i \int_{\Omega^\varepsilon}  |c_j^{\varepsilon}||c_i^{\varepsilon}| \dx.
\end{multline*}
We now estimate the right-hand side. The bulk reaction term is estimated directly using the vector 
norm rather than a component-wise estimate, giving 
$\sum_{i,j=1}^3 K_i\kappa_i^j\int_{\Omega^\varepsilon}|c_j^\varepsilon||c_i^\varepsilon|\,dx 
\le C\|\vec c_\varepsilon\|^2_{L^2(\Omega^\varepsilon)}$. 
For the boundary integral over $\Gamma_N^\varepsilon$, we split the integrand as 
$|q_{\mathrm{prod},i}||c_i^\varepsilon| \le |q_{\mathrm{prod},i}|^2 + |c_i^\varepsilon|^2$. The 
resulting data term is bounded by Assumption \ref{asprod} together with the uniform 
surface-measure estimate $\varepsilon|\Gamma_N^\varepsilon| \le C$, while the resulting term, 
$\varepsilon\|c_i^\varepsilon\|^2_{L^2(\Gamma_N^\varepsilon)}$, is estimated via the scaled trace 
inequality \eqref{scaledtrace}, with a constant independent of $\varepsilon$. The integral over 
$\Gamma_C^\varepsilon$ is estimated analogously. Altogether, this yields
\begin{equation*}
\frac{d}{dt}\|\vec{c}_\varepsilon\|_{L^2(\Omega^\varepsilon)}^2 +\| \vec{c}_\varepsilon\|_{W^{1,2}(\Omega^\varepsilon)}^2 
\leq C\left(\|\vec{q}_{\mathrm{prod}}\|^2_{L^\infty((0,T)\times\Gamma_N^\varepsilon)} +\|\vec{r}_{\varepsilon}\|^2_{L^\infty((0,T)\times \Gamma_C^\varepsilon)} \right) 
 +  \delta \|\vec{c}_\varepsilon\|_{W^{1,2}(\Omega^\varepsilon)}^2 + C(\delta) \|\vec{c}_\varepsilon\|_{L^2(\Omega^\varepsilon)}^2
\end{equation*}
for arbitrary $\delta>0$. Choosing $\delta$ sufficiently small and applying Gronwall’s lemma establishes
\begin{equation}\label{bulkestimates}
    \|\vec{c}_\varepsilon\|_{L^\infty(L^2(\Omega^\varepsilon))} + \|\vec{c}_\varepsilon\|_{L^2(W^{1,2}(\Omega^\varepsilon))} \leq C_1,
\end{equation}
where the constant $C_1$ doesn't depend on $\varepsilon$ nor the choice of $\vec{d}_\varepsilon$.
Moreover, \eqref{bulkestimates} implies the bound \[\|\dert \vec{c}_\varepsilon\|_{L^2(W^{-1,2}(\Omega^\varepsilon))} \leq C\] by standard arguments, where $C$ is again independent of $\varepsilon$ and the choice of $\vec{d}_\varepsilon$.  These a priori estimates ensure the existence of a solution $\vec{c}_\varepsilon$ to \eqref{wftkbulk} via a Galerkin approximation.

Let us now show that any solution of \eqref{wftkbulk} is nonnegative. 
We test \eqref{wftkbulk} with $\varphi_i =c_i^{\varepsilon,-}$. The sign of the degradation function is given by Assumption \ref{asdeg}, and it thus provides an $L^2$ estimate on the bulk concentration. The boundary terms are nonpositive by the choice of coefficients made in Assumptions \ref{asbind} and \ref{asprod}, and the already established nonnegativity of $r^\varepsilon_j$ for all $j$. We have
\begin{equation}\label{nonnegativeck}
\begin{aligned}
    \frac{1}{2}\frac{d}{dt} \|c_i^{\varepsilon,-}\|_{L^2(\Omega^\varepsilon)}^2 +  & D_i\|\nabla c_i^{\varepsilon,-}\|_{ L^{2}(\Omega^\varepsilon)}^2   +  \delta_i\| c_i^{\varepsilon,-}\|_{ L^{2}(\Omega^\varepsilon)}^2\\&\leq \int_{\Omega^\varepsilon} f_{\mathrm{agg},i}(\vec{c}_{\varepsilon,+}, \vec{c}_{\varepsilon,+})c_i^{\varepsilon,-}  \dx \leq \int_{\Omega^\varepsilon} \left(\sum_{l=1}^3 \kappa^j_{il} c_l^{\varepsilon,+} 
        + \kappa^j_i \right)c_i^{\varepsilon,+}c_i^{\varepsilon,-} \dx = 0
,\end{aligned}
\end{equation}
where the inequality follows from the nonnegativity of $\kappa_{jl}^i$ and $\kappa_j^i$ for $i\ne j$ given by Assumption~\ref{asagg}. 
Invoking Assumption~\ref{asinit} together with Gronwall’s lemma, we deduce \[\|c_i^{\varepsilon,-}\|_{L^\infty(L^2(\Omega^\varepsilon))}^2 + \|c_i^{\varepsilon,-}\|_{L^2(W^{1,2}(\Omega^\varepsilon))}^2\leq 0,\]
which implies $c_i^\varepsilon \geq 0$ almost everywhere on $\Omega^\varepsilon \times (0,T)$ for all $i$.

Hence, selecting $ C_{ \mathcal{A} }\geq C_1$ in \eqref{seta}, with $C_1$ as in \eqref{bulkestimates}, ensures $F(\mathcal{A})\subset \mathcal{A}$. 
By the Rellich-Kondrachov theorem (see Theorem 6.3 in \cite{Adams2003}) we have the compact embedding \(W^{1,2}(\Omega^\varepsilon) \hookrightarrow \hookrightarrow W^{r,2}(\Omega^\varepsilon)\).
Hence, applying the Aubin--Lions lemma \cite{Si1987}, 
\[
\left\{\vec{c}_\varepsilon \in L^2(0,T;W^{1,2}(\Omega^\varepsilon))^3 \; : \; \partial_t \vec{c}_\varepsilon \in L^2(0,T;W^{-1,2}(\Omega^\varepsilon))^3 \right\}
\hookrightarrow \hookrightarrow
L^2(0,T;W^{r,2}(\Omega^\varepsilon))^3,
\quad r < 1,
\]
 implies that $F$ is continuous and compact.
Since $\mathcal{A}$ is convex, closed, and bounded in $L^2(0,T;W^{r,2}(\Omega^\varepsilon))^3$, Schauder’s fixed point theorem (see e.g. \cite{brezis2011}) guarantees the existence of a fixed point $F(\vec{c}_\varepsilon) = \vec{c}_\varepsilon$. This fixed point provides a weak solution to the truncated system \eqref{cutISFeps},\eqref{productioneps},\eqref{uptakeeps}, \eqref{cutreceptoreps}. Moreover, this solution fulfills the uniform estimates \eqref{estimatec}--\eqref{estimateder}.

Owing to the nonnegativity of $r_j^\varepsilon$ and $c_i^\varepsilon$, together with the uniform $L^\infty$ bounds on $r_j^\varepsilon$, one may choose $k>0$ sufficiently large such that the truncation in \eqref{cutreceptoreps} is never attained. Consequently, the solution of the approximate problem \eqref{cutISFeps}, \eqref{productioneps}, \eqref{uptakeeps}, \eqref{cutreceptoreps} coincides with the solution of the original system \eqref{ISFeps}--\eqref{receptoreps}.
\end{proof}


\begin{lemma}\label{boundedws}
    Under Assumptions \ref{asagg}--\ref{asinit}, the weak solutions of \eqref{ISFeps}--\eqref{receptoreps}
     satisfy $c_i^\varepsilon\in L^\infty(Q^\varepsilon)$. Moreover, if the solution fulfills \eqref{estimatec}--\eqref{estimater}, then the improved estimate \eqref{estimatecbetter} holds.
\end{lemma}
\begin{proof}
First, let us note that the embedding
\[
L^\infty(0,T;L^2(\Omega^\varepsilon)) \cap L^2(0,T;W^{1,2}(\Omega^\varepsilon)) \hookrightarrow L^{\frac{10}{3}}(Q^\varepsilon)
\]
together with the regularity \eqref{wfspace} required in Definition~\ref{weaksoldef} already guarantees that
\begin{align}\label{startalikakosweak}
\|c_i^\varepsilon\|_{L^{\frac{10}{3}}(Q^\varepsilon)} \le C(\varepsilon)
\end{align}
for every weak solution and each $i=1,2,3$, with a constant that a priori depends on $\varepsilon$. If, in addition, the solution satisfies the uniform estimate \eqref{estimatec}, then 
\begin{align}\label{startalikakos}
      \|c_i^\varepsilon\|_{L^{\frac{10}{3}}(Q^\varepsilon)} \le C
\end{align}
for all $i$, uniformly in $\varepsilon$. 
We improve this bound on $c_i^\varepsilon$ using a Moser--Alikakos iteration \cite{Alikakos1979}.
The main idea is to iteratively increase the integrability exponent by testing the weak formulation \eqref{wfbulk} with powers of the solution.

More precisely, for $p > 1$, we take $(c_i^\varepsilon)^{2p-1}$ as a test function in \eqref{wfbulk} and sum over $i$. Invoking Assumptions \ref{asagg}--\ref{asdeg}, 
 applying Hölder's inequality to the aggregation term and estimating the boundary terms by the boundedness of $\vec{q}_{\mathrm{prod}}$ and $\vec{r}_\varepsilon$, we obtain
\begin{align*}
\sum_{i=1}^3 &\left(\frac{1}{2p}\frac{d}{dt}\|c_i^\varepsilon\|_{L^{2p}(\Omega^\varepsilon)}^{2p} + \frac{2p-1}{p^2}D_i\|\nabla (c_i^\varepsilon)^p\|_{L^2(\Omega^\varepsilon)}^2 +  \delta_i\|c_i^\varepsilon\|_{ L^{2p}(\Omega^\varepsilon)}^{2p}\right) \\
\leq &C\left(\sum_{i,j,k=1}^3\|c_j^\varepsilon\|_{L^{1+2p}(\Omega^\varepsilon)}\|c_k^\varepsilon\|_{L^{1+2p}(\Omega^\varepsilon)}\|c_i^\varepsilon\|_{L^{1+2p}(\Omega^\varepsilon)}^{2p-1} + \sum_{i,j=1}^3\|c_j^\varepsilon\|_{L^{2p}(\Omega^\varepsilon)}\|c_i^\varepsilon\|_{L^{2p}(\Omega^\varepsilon)}^{2p-1}\right) \\
&+\varepsilon C \sum_{i=1}^3\left(\|\vec{q}_{\mathrm{prod}}\|_{L^{\infty}(\Gamma_N^\varepsilon \times (0,T))} \|c_i^\varepsilon\|_{L^{2p-1}(\Gamma^\varepsilon_N)}^{2p-1} + \|\vec{r}_\varepsilon\|_{L^{\infty}(\Gamma_C^\varepsilon \times (0,T))} \left(\|c_i^\varepsilon\|_{L^{2p-1}(\Gamma^\varepsilon_C)}^{2p-1} + \|c_i^\varepsilon\|_{L^{2p}(\Gamma^\varepsilon_C)}^{2p}\right)\right).
\end{align*}
 Note that, for any measurable $\Gamma^\varepsilon \subseteq \Gamma_C^\varepsilon \cup \Gamma_N^\varepsilon$, it holds 
 \[\varepsilon \|c_i^\varepsilon\|_{L^{2p-1}(\Gamma^\varepsilon)}^{2p-1} \leq \varepsilon\|c_i^\varepsilon\|_{L^{2p}(\Gamma^\varepsilon)}^{2p} + 1  = \varepsilon\|(c_i^\varepsilon)^p\|_{L^{2}(\Gamma^\varepsilon)}^{2} + 1. \] 
 In particular, this estimate applies to \(\Gamma_C^\varepsilon\) and \(\Gamma_N^\varepsilon\).
 Furthermore, noticing that \(1 \geq \frac{2p-1}{p^2} \geq \frac{1}{p}\), using the nonnegativity of coefficients given by Assumptions \ref{asdeg} and \ref{asdif}, applying Young's inequality on the bulk reaction terms, and exploiting the uniform bounds on $\vec{q}_{\mathrm{prod}}$ given by \ref{asprod}, we arrive at
\begin{equation}\label{alikakoshelp}
    \begin{aligned}
\frac{1}{p}\sum_{i=1}^3& \left(\frac{d}{dt}\|c_i^\varepsilon\|_{L^{2p}(\Omega^\varepsilon)}^{2p} + \|\nabla (c_i^\varepsilon)^p\|_{L^2(\Omega^\varepsilon)}^2 +  \|c_i^\varepsilon\|_{ L^{2p}(\Omega^\varepsilon)}^{2p}\right) \\
&\leq C  \left(1 + \|\vec{r}_\varepsilon\|_{L^{\infty}(\Gamma_C^\varepsilon \times (0,T))}\right) \sum_{i=1}^3 \left(\|c_i^\varepsilon\|_{L^{1+2p}(\Omega^\varepsilon)}^{1+2p} + \|c_i^\varepsilon\|_{L^{2p}(\Omega^\varepsilon)}^{2p} + \varepsilon\|(c_i^\varepsilon)^p\|_{L^{2}(\Gamma_C^\varepsilon \cup \Gamma_N^\varepsilon)}^{2}  + 1\right).
\end{aligned}
\end{equation}
Note that the expression
\[
C\left(1+\|\vec r_\varepsilon\|_{L^\infty(\Gamma_C^\varepsilon\times(0,T))}\right)
\]
is bounded for each fixed \(\varepsilon\) by \eqref{wfspace}, and is moreover uniformly bounded with respect to \(\varepsilon\) whenever the estimate \eqref{estimater} holds. Since this constant is independent of \(p\), we absorb it, together with the other constants independent of \(p\) and \(\varepsilon\), into a generic constant \(C\), suppressing the explicit dependence on \(\vec{r}_\varepsilon\). We emphasize that \(C\) is independent of \(\varepsilon\) under \eqref{estimater}, whereas without \eqref{estimater} it may depend on \(\varepsilon\) through the bound provided by \eqref{wfspace}.
We now estimate the boundary terms. 
 The scaled trace inequality \eqref{scaledtrace} implies
 \[ \varepsilon \|(c_i^\varepsilon)^p\|_{L^{2}(\Gamma^\varepsilon)}^{2} \leq p C_\delta \|(c_i^\varepsilon)^p\|_{L^{2}(\Omega^\varepsilon)}^{2} +  \varepsilon^2 \frac{\delta}{p}\|\nabla (c_i^\varepsilon)^p\|_{L^{2}(\Omega^\varepsilon)}^{2} \] 
 for arbitrary $\delta>0$. Choosing $\delta$ fixed and sufficiently small, the gradient part is absorbed by the second term on the left-hand side of \eqref{alikakoshelp}. 
Multiplying \eqref{alikakoshelp} by $p$, integrating in time until $\tau\in(0,T]$, invoking Assumption~\ref{asinit} on the uniform boundedness of initial data, and estimating the component-wise norm by the vector norm, we obtain 
\begin{align*}
\sum_{i=1}^3 \left(\|c_i^\varepsilon\|_{L^{2p}(\Omega^\varepsilon)}^{2p}(\tau) + \int_0^\tau \| (c_i^\varepsilon)^p\|_{W^{1,2}(\Omega^\varepsilon)}^2\dt \right)
\leq p^2C\sum_{i=1}^3\left(1 + \|\vec{c}_\varepsilon\|_{L^{1+2p}(Q^\varepsilon)}^{1+2p} + \|\vec{c}_\varepsilon\|_{L^{2p}(Q^\varepsilon)}^{2p} \right). 
\end{align*}
Taking the supremum over $\tau\in(0,T]$ and noting that the term $\|\vec{c}_\varepsilon\|_{L^{2p}(Q^\varepsilon)}^{2p}$ is controlled by $\|\vec{c}_\varepsilon\|_{L^{2p+1}(Q^\varepsilon)}^{2p+1} + 1$, we obtain the iterative estimate
\begin{align}\label{alikakosit}
\sum_{i=1}^3 \left(\|c_i^\varepsilon\|_{L^\infty(L^{2p}(\Omega^\varepsilon))}^{2p} +\| (c_i^\varepsilon)^p\|_{L^2(W^{1,2}(\Omega^\varepsilon))}^2\right) 
\leq p^{2}C\left(1  + \|\vec{c}_\varepsilon\|_{L^{1+2p}(Q^\varepsilon)}^{1+2p}\right).
\end{align}

Setting \(u_i := (c_i^\varepsilon)^p\), invoking the 
embedding \(L^\infty(0,T;L^2(\Omega^\varepsilon)) \cap L^2(0,T;W^{1,2}(\Omega^\varepsilon))\hookrightarrow L^{\frac{10}{3}}(Q^\varepsilon),\) together with the estimate \eqref{alikakosit}, 
we obtain\[\|c_i^\varepsilon\|_{L^{\frac{10}{3}p}(Q^\varepsilon)}^p=\|u_i\|_{L^{\frac{10}{3}}(Q^\varepsilon)} \le C\left(\|u_i\|_{L^\infty(L^{2}(\Omega^\varepsilon))} + \| u_i\|_{L^2(W^{1,2}(\Omega^\varepsilon))}\right) \le pC\left(1+\|\vec c_\varepsilon\|_{L^{1+2p}(Q^\varepsilon)}^{\frac{1}{2}+p}\right), \]
for each $i$. Taking the $p$-root yileds
\begin{align}\label{recursion}
    \|c_i^\varepsilon\|_{L^{\frac{10}{3}p}(Q^\varepsilon)} \le p^{1/p}C\left(1+\|\vec c_\varepsilon\|_{L^{1+2p}(Q^\varepsilon)}^{\frac{1}{2p}+1}\right),
\end{align}
for each $i$.

To bootstrap the estimate, we choose the exponents \((p_k)_{k\ge0}\) recursively so that the norm on the right-hand side is controlled by the norm on the left-hand side from the previous step. 
By \eqref{startalikakosweak} we have a bound on $\|\vec c_\varepsilon\|_{L^{1+2p}(Q^\varepsilon)}$ whenever $1+2p\leq 10/3$. If in addition \eqref{estimatec} holds, this bound is uniform in $\varepsilon$ by \eqref{startalikakos}. 
Taking \(p_0=\frac{7}{6}\) fulfills $1+2p = 10/3$, so the right-hand side of \eqref{recursion} is bounded. Furthermore, if \eqref{estimatec} and \eqref{estimater} hold, it is uniformly bounded in $\varepsilon$. Consequently, by \eqref{recursion} applied with $p=p_0$, the norm
\[\|c_i^\varepsilon\|_{L^{\frac{10}{3}p_0}(Q^\varepsilon)} = \|c_i^\varepsilon\|_{L^{\frac{35}{9}}(Q^\varepsilon)} \]
is bounded as well, which initializes the bootstrap. 

We then define $(p_k)_{k\ge1}$ recursively by
\begin{align}\label{recursion1}
    1+2p_{k}=\frac{10}{3}p_{k-1}.
\end{align}
With this choice, the exponent on the right-hand side of \eqref{recursion} at step $k$ coincides with the exponent obtained on the left-hand side at step $k-1$, yielding
\[
\|c_i^\varepsilon\|_{L^{\frac{10}{3}p_k}(Q^\varepsilon)}
\le
p_k^{1/p_k}C
\left(
1+
\|\vec c_\varepsilon\|_{L^{\frac{10}{3}p_{k-1}}(Q^\varepsilon)}^{1+\frac1{2p_k}}
\right).
\]
Iterating this bound therefore provides control in successively higher Lebesgue spaces. Solving the recursion \eqref{recursion1} explicitly,
\[
p_k=\frac34+\frac14\left(\frac53\right)^{k+1},
\qquad k\ge0,
\]
so that $p_k\to\infty$ exponentially fast as $k\to\infty$. Since $p_k\ge p_0=\tfrac{7}{6}>1$ for all $k$, the elementary bound $\sup_{x>1}x^{1/x}=e^{1/e}$ shows that $p_k^{1/p_k}$ is bounded uniformly in $k$. The classical Moser--Alikakos iteration \cite{Alikakos1979} therefore yields
\[
c_i^\varepsilon\in L^\infty(Q^\varepsilon).
\]
Furthermore, the resulting $L^\infty$-bound is uniform in $\varepsilon$ whenever \eqref{estimatec}--\eqref{estimater} hold.

Since $c_i^\varepsilon\in L^2(0,T;W^{1,2}(\Omega^\varepsilon))$, we can conclude that it is also bounded in $L^\infty$ on $\Gamma_C^\varepsilon$ and estimate \eqref{estimatecbetter} indeed holds.
\end{proof}

Note that due to Lemma~\ref{boundedws}, the solution constructed in Lemma~\ref{existencews} fulfills the uniform estimates \eqref{estimatec}--\eqref{estimatecbetter}.

\begin{lemma}\label{uniquews}
    Under Assumptions \ref{asagg}--\ref{asinit}, the weak solution to the microscopic problem \eqref{ISFeps}--\eqref{receptoreps} is unique.
\end{lemma}
\begin{proof} 
Let $(\vec{c}_\varepsilon, \vec{r}_\varepsilon)$ and $(\vec{d}_\varepsilon, \vec{s}_\varepsilon)$ be two weak solutions to the microscopic problem \eqref{ISFeps}--\eqref{receptoreps} with the same initial data. By Lemma~\ref{boundedws} we have $c_i^\varepsilon, d_i^\varepsilon \in L^\infty(Q^\varepsilon)$.
We define the differences of the solutions as $u^\varepsilon_i = c^\varepsilon_i - d^\varepsilon_i$ and $v^\varepsilon_j = r^\varepsilon_j - s^\varepsilon_j$. 

We test \eqref{wfbulk} for $c^\varepsilon_i$ and $d^\varepsilon_i$, $i=1,2,3$, by $u^\varepsilon_i$ and subtract the latter from the former. Similarly, we test \eqref{wfsurf} for $r^\varepsilon_j$ and $s^\varepsilon_j$, $j=0,\dots,3$, by $\varepsilon v^\varepsilon_j$ and subtract. 
Summing over $i$ and $j$ and adding the bulk and surface equations together, we obtain
\begin{multline*}
    \varepsilon\frac{d}{dt} \|\vec{v}_\varepsilon\|_{L^2(\Gamma_C^\varepsilon)}^2 +  \frac{d}{dt} \|\vec{u}_\varepsilon\|_{L^2(\Omega^\varepsilon)}^2 +  \|\vec{u}_\varepsilon\|_{W^{1,2}(\Omega^\varepsilon)}^2  \\
    \leq C\left| \sum_{j=0}^{3} \varepsilon\int_{\Gamma_C^\varepsilon} (f_{\mathrm{rec},j}(\vec{c}_\varepsilon, \vec{r}_\varepsilon) - f_{\mathrm{rec},j}(\vec{d}_\varepsilon, \vec{s}_\varepsilon)) v^\varepsilon_j \, dS \right|
+ C \left| \sum_{i=1}^{3} \varepsilon \int_{\Gamma_C^\varepsilon} (f_{\mathrm{bind},i}(\vec{c}_\varepsilon, \vec{r}_\varepsilon) - f_{\mathrm{bind},i}(\vec{d}_\varepsilon, \vec{s}_\varepsilon)) u^\varepsilon_i \, dS\right| \\
+ C \left|\sum_{i=1}^{3} \int_{\Omega^\varepsilon} (f_{\mathrm{agg},i}(\vec{c}_\varepsilon,\vec{c}_\varepsilon) - f_{\mathrm{agg},i}(\vec{d}_\varepsilon, \vec{d}_\varepsilon)) u^\varepsilon_i \, dx \right| = C\left(|I_1| + |I_2| + |I_3|\right).
\end{multline*}

We now estimate each term separately. For the bulk aggregation contribution, we obtain
\[ |I_3| \leq \sum_{i=1}^3 C \int_{\Omega^\varepsilon} \left( \left|\sum_{k,l=1}^3 c^\varepsilon_l c^\varepsilon_k - d^\varepsilon_l d^\varepsilon_k\right| |u^\varepsilon_i| + \sum_{k=1}^3 |u^\varepsilon_k|  |u^\varepsilon_i| \right) dx. \]
Using the identity $c^\varepsilon_l c^\varepsilon_k - d^\varepsilon_l d^\varepsilon_k = u^\varepsilon_lc^\varepsilon_k + d^\varepsilon_lu^\varepsilon_k$ and relabeling indices, we deduce that
\[ |I_3| \leq \sum_{i,k,l=1}^3 C \int_{\Omega^\varepsilon} |u^\varepsilon_l|\left(|c^\varepsilon_k| + |d^\varepsilon_k|\right) |u^\varepsilon_i| \, dx + C \sum_{i,k=1}^3 \|u^\varepsilon_k\|_{L^2(\Omega^\varepsilon)} \|u^\varepsilon_i\|_{L^2(\Omega^\varepsilon)}. \]
Applying Hölder's inequality with exponents $2, 6, 3$, the Sobolev embedding, and Young's inequality, we can estimate
\[ \sum_{i,k,l=1}^3 C \int_{\Omega^\varepsilon} |u^\varepsilon_l| \left( |c^\varepsilon_k| + |d^\varepsilon_k| \right) |u^\varepsilon_i| \, dx \leq \delta \sum_{i=1}^3 \|u^\varepsilon_i\|_{W^{1,2}(\Omega^\varepsilon)}^2 + C(\delta) \sum_{k,l=1}^3 \left( \|c^\varepsilon_k\|_{W^{1,2}(\Omega^\varepsilon)}^2 + \|d^\varepsilon_k\|_{W^{1,2}(\Omega^\varepsilon)}^2 \right) \|u^\varepsilon_l\|_{L^2(\Omega^\varepsilon)}^2, \]
for $\delta>0$ arbitrarily small. 
Hence
\[ |I_3| \leq  \delta \|\vec{u}_\varepsilon\|^2_{W^{1,2}(\Omega^\varepsilon)} + C(\delta) \left( 1 +  \|\vec{c}_\varepsilon\|_{W^{1,2}(\Omega^\varepsilon)}^2 + \|\vec{d}_\varepsilon\|_{W^{1,2}(\Omega^\varepsilon)}^2 \right) \|\vec{u}_\varepsilon\|^2_{L^2(\Omega^\varepsilon)}. \]

For the surface reactions, we have
    \[ |I_1| \leq \varepsilon C \sum_{j,k=0}^3 \int_{\Gamma_C^\varepsilon}\sum_{l=1}^3  \left(|c^\varepsilon_l||v_k^\varepsilon|+  |u^\varepsilon_l||s^\varepsilon_k|\right) |v^\varepsilon_j| \, dS + \varepsilon C \sum_{j,k=0}^3 \|v^\varepsilon_k\|_{L^2(\Gamma_C^\varepsilon)} \|v^\varepsilon_j\|_{L^2(\Gamma_C^\varepsilon)}. \] 
    Using Hölder's inequality and estimating the component-wise norm by the vector norm, we have
    \[ |I_1| \leq  \varepsilon C \left( \|\vec{c}_\varepsilon\|_{L^\infty((0,T)\times \Gamma_C^\varepsilon)} \|\vec{v}_\varepsilon\|^2_{L^2(\Gamma_C^\varepsilon)} +  \|\vec{v}_\varepsilon\|_{L^2(\Gamma_C^\varepsilon)} \|\vec{u}_\varepsilon\|_{L^2(\Gamma_C^\varepsilon)} \|\vec{s}_\varepsilon\|_{L^\infty((0,T)\times \Gamma_C^\varepsilon)} +  \|\vec{v}_\varepsilon\|^2_{L^2(\Gamma_C^\varepsilon)}\right). \] 
    Young's inequality and the Sobolev embedding imply that
    \[ |I_1| \leq  \varepsilon C(\delta) \left(1 +  \|\vec{c}_\varepsilon\|^2_{L^\infty((0,T)\times \Gamma_C^\varepsilon)} +   \|\vec{s}_\varepsilon\|^2_{L^\infty((0,T)\times \Gamma_C^\varepsilon)}\right)\|\vec{v}_\varepsilon\|^2_{L^2(\Gamma_C^\varepsilon)} + \delta \|\vec{u}_\varepsilon\|^2_{W^{1,2}(\Omega^\varepsilon)} . \] 
 Similarly, for the surface binding term, we obtain  
     \[ |I_2| \leq   \varepsilon C(\delta) \left(1 +  \|\vec{c}_\varepsilon\|^2_{L^\infty((0,T)\times \Gamma_C^\varepsilon)} + \|\vec{s}_\varepsilon\|^2_{L^\infty((0,T)\times \Gamma_C^\varepsilon)} \right)\|\vec{v}_\varepsilon\|^2_{L^2(\Gamma_C^\varepsilon)} + \delta \|\vec{u}_\varepsilon\|^2_{W^{1,2}(\Omega^\varepsilon)}.\] 

Choosing $\delta$ small enough, we arrive at the estimate
\begin{multline*}
    \frac{d}{dt} \left(\varepsilon\|\vec{v}_\varepsilon\|_{L^2(\Gamma_C^\varepsilon)}^2 +  \|\vec{u}_\varepsilon\|_{L^2(\Omega^\varepsilon)}^2\right) +  \|\vec{u}_\varepsilon\|_{W^{1,2}(\Omega^\varepsilon)}^2 \\ \leq C \left( 1 +  \|\vec{c}_\varepsilon\|_{W^{1,2}(\Omega^\varepsilon)}^2 + \|\vec{d}_\varepsilon\|_{W^{1,2}(\Omega^\varepsilon)}^2 + \|\vec{c}_\varepsilon\|^2_{L^\infty((0,T)\times \Gamma_C^\varepsilon)} +   \|\vec{s}_\varepsilon\|^2_{L^\infty((0,T)\times \Gamma_C^\varepsilon)} \right) \left(\|\vec{u}_\varepsilon\|^2_{L^2(\Omega^\varepsilon)} + \varepsilon \|\vec{v}_\varepsilon\|^2_{L^2(\Gamma_C^\varepsilon)}\right).
\end{multline*}
Using Gronwall's lemma, the integrability of terms on the right-hand side, and the fact that \[\|\vec{v}_\varepsilon(0)\|_{L^2(\Gamma_C^\varepsilon)}^2 = 0 = \|\vec{u}_\varepsilon(0)\|_{L^2(\Omega^\varepsilon)}^2,\] we have  
\[\varepsilon\|\vec{v}_\varepsilon\|_{L^2(\Gamma_C^\varepsilon)}^2 +  \|\vec{u}_\varepsilon\|_{L^2(\Omega^\varepsilon)}^2 +  \|\vec{u}_\varepsilon\|_{W^{1,2}(\Omega^\varepsilon)}^2 \leq 0,\]
which implies that $\vec{c}_\varepsilon = \vec{d}_\varepsilon, $ almost everywhere on $\Omega^\varepsilon\times(0,T)$, and $\vec{r}_\varepsilon = \vec{s}_\varepsilon, $ almost everywhere on $\Gamma_C^\varepsilon\times(0,T)$.
\end{proof}

\section{Convergence results and derivation of macroscopic equations.}\label{sec:conv}
In this section, we use the a priori estimates of Theorem~\ref{wellposednes} 
to pass to the homogenization limit and derive an effective macroscopic problem. 
More precisely, we show that solutions to the microscopic 
problem \eqref{ISFeps}--\eqref{receptoreps} converge, up to a subsequence, to a solution of the macroscopic problem \eqref{macroc}--\eqref{cellproblem}. The derivation relies on the 
theory of two-scale convergence, for which the necessary definitions and results 
are collected in Appendix~\ref{homgtheory}.

Recall the definition of $Z$, $Y$, $\Gamma_C$, and $\Gamma_N$ from the beginning of 
Section~\ref{sec:micro}. For the two-scale limit functions, we will need the following periodic 
function spaces. Let $\Cper(Z)$ denote the space of continuous functions on $\mathbb{R}^3$ that are 
$l_i$-periodic in the $i$-th variable, and let $\Sper(Z)$ denote the space of smooth $Z$-periodic 
functions. We define $W^{1,2}_{\mathrm{per}}(Z)$ as the completion of $\Sper(Z)$ with respect to 
the $W^{1,2}$-norm. Analogously, let $\Cper(\Gamma_C)$ denote the space of continuous functions on 
$\bigcup_{\xi \in \mathbb{Z}^3} (\Gamma_C + \xi)$ that are $l_i$-periodic in the $i$-th variable.

Since the theory of two-scale convergence in Appendix~\ref{homgtheory} is 
formulated for sequences defined on the fixed macroscopic domain $\Omega$, 
we employ the extensions $\tilde{c}_i^\varepsilon \in L^2(0,T; W^{1,2}(\Omega))$ 
introduced in Remark~\ref{rem:extension}. In addition, we consider the extension 
of $c_i^\varepsilon$ from $\Omega^\varepsilon$ to $\Omega$ by zero, denoted 
by $\hat{c}_i^\varepsilon$. The two extensions are related by
\begin{equation*}
    \hat{c}_i^\varepsilon = \tilde{c}_i^\varepsilon \chi^\varepsilon,
\end{equation*}
where $\chi^\varepsilon$ is the characteristic function of $\Omega^\varepsilon$, 
defined as the rescaled version of the characteristic function $\chi_Y$ of 
the extracellular subdomain $Y \subset Z$, extended $Z$-periodically to 
$\mathbb{R}^3$. More precisely, we set
\begin{equation*}
    \chi^\varepsilon(x) := \chi_Y\!\left(\frac{x}{\varepsilon}\right), \qquad x \in \Omega.
\end{equation*}

We begin by extracting weakly and two-scale convergent subsequences from the 
uniform estimates of Theorem~\ref{wellposednes}, and subsequently identify the two-scale limits to close the homogenized system.

\begin{lemma}[Two-scale and weak convergences]\label{lem:twoscale}
Under Assumptions~\ref{asagg}--\ref{asinit}, there exist limit functions
\[
c_i \in L^2(0,T; W^{1,2}(\Omega)), 
\qquad
c_i^1 \in L^2\bigl((0,T)\times\Omega; W^{1,2}_\mathrm{per}(Z)/\mathbb{R}\bigr),
\qquad
r_j \in L^\infty\bigl((0,T)\times\Omega\times\Gamma_C\bigr),
\]
for $i=1,2,3$ and $j=0,1,2,3$, such that, up to a subsequence, the following 
convergences hold as $\varepsilon \to 0$:
\begin{align}
    \tilde{c}_i^\varepsilon &\rightharpoonup c_i &&\text{weakly in } L^2(0,T; W^{1,2}(\Omega)), \label{conv:cweak}\\
    \nabla \tilde{c}_i^\varepsilon &\rightharpoonup \nabla_x c_i + \nabla_y c_i^1,  &&\text{two-scale, } \label{conv:gradc}\\
    \hat{c}_i^\varepsilon &\rightharpoonup c_i\chi_Y, &&\text{two-scale, } \label{conv:c0weak}\\
    \nabla \hat{c}_i^\varepsilon &\rightharpoonup \left(\nabla_x c_i + \nabla_y c_i^1\right)\chi_Y, &&\text{two-scale, } \label{conv:gradc0}\\
    r_j^\varepsilon &\rightharpoonup r_j &&\text{two-scale, } \label{conv:r}\\
   \partial_t r_j^\varepsilon &\rightharpoonup \partial_t r_j  &&\text{two-scale. } \label{conv:dtr}
\end{align}
Moreover, $c_i \geq 0$ a.e.\ in $(0,T)\times\Omega$ and 
$r_j \geq 0$ a.e.\ on $(0,T)\times\Omega\times\Gamma_C$.
\end{lemma} 
\begin{proof}
The estimate \eqref{estimatec} from Theorem \ref{wellposednes} implies the convergence \eqref{conv:cweak}. Estimate \eqref{estimatec}, Remark \ref{rem:extension}, and the second part of Theorem \ref{thm:cpcttwoscale} yield the convergence \eqref{conv:gradc}. Moreover, estimate \eqref{estimatec} together with the third part of Theorem \ref{thm:cpcttwoscale}, implies convergences \eqref{conv:c0weak} and \eqref{conv:gradc0}.
By the estimate \eqref{estimater} and Theorem \ref{def:manifoldtwoscale}, there exists a subsequence of $r_j^\varepsilon$ that two-scale converges to $r_j\in L^2\bigl((0,T)\times\Omega\times\Gamma_C\bigr)$. 

To upgrade this regularity, we show that the limit is essentially bounded. Indeed, by \eqref{estimater} we have $|r_j^\varepsilon(t,x)| \le C$ a.e.\ on $(0,T)\times \Gamma_C^\varepsilon$. Let $\varphi \in L^1(0,T; C^0(\Omega; C_{\mathrm{per}}(\Gamma_C)))$, $\varphi \ge 0$. Then
\[
|r_j^\varepsilon(t,x)\varphi(t,x,x/\varepsilon)| \le C\, \varphi(t,x,x/\varepsilon)
\]
almost everywhere, and therefore, after multiplication by $\varepsilon$ and integration over $(0,T)\times \Gamma_C^\varepsilon$, the sequence is uniformly bounded. Passing to the limit in the sense of Theorem \ref{def:manifoldtwoscale} yields
\[
\left|\int_0^T \!\!\int_{\Omega\times\Gamma} r_j(t,x,y)\varphi(t,x,y)\,dS_y\,dx\,dt\right|
\le C \int_0^T \!\!\int_{\Omega\times\Gamma} \varphi(t,x,y)\,dS_y\,dx\,dt,
\]
for all nonnegative $\varphi$. By a standard localization argument, this implies $|r_j(t,x,y)| \le C$ almost everywhere in $(0,T)\times \Omega\times \Gamma_C$, and hence $r_j \in L^\infty\bigl((0,T)\times\Omega\times\Gamma_C\bigr)$.
 
The convergence \eqref{conv:dtr} is due to the estimate \eqref{estimateder} and Theorem \ref{def:manifoldtwoscale}.
The nonnegativity of the limits is preserved by weak convergence.
\end{proof}

To pass to the limit in the nonlinear terms, weak convergence alone is insufficient. The following lemma provides the necessary strong convergence.

\begin{lemma}[Strong convergence]\label{lem:strong}
Under Assumptions~\ref{asagg}--\ref{asinit}, the following strong convergences 
hold up to a subsequence as $\varepsilon \to 0$: 
\begin{align}
    \tilde{c}_i^\varepsilon &\to c_i \quad \text{strongly in } L^p((0,T)\times\Omega), p \in \left[1,\infty\right) 
    \label{sconv:c}\\
  c_i^\varepsilon &\to c_i \quad \text{strongly two-scale on } (0,T)\times\Gamma_C.
    \label{sconv:trc}
\end{align}
\end{lemma}
\begin{proof}
The strong convergence \eqref{sconv:c} for $p\leq 2$ follows from the uniform bounds 
\eqref{estimatec} and \eqref{estimateder} via the Aubin--Lions lemma  \cite{Si1987}, which yields compactness in $L^2((0,T)\times\Omega)$ 
from the uniform bounds on $c_i^\varepsilon$ in $L^2(0,T; W^{1,2}(\Omega^\varepsilon))$ 
and on $\partial_t c_i^\varepsilon$ in $L^2(0,T; W^{-1,2}(\Omega^\varepsilon))$. For $p\geq 2 $, the result follows by interpolation between the strong
convergence in \(L^2\) and the uniform boundedness in \(L^{\infty}\).

In addition, using the Rellich--Kondrachov theorem
\cite{Adams2003},
\[
W^{1,2}(\Omega)\hookrightarrow\hookrightarrow W^{r,2}(\Omega),
\qquad r\in(1/2,1),
\]
and applying the Aubin--Lions compactness argument to the extensions
\(\tilde c_i^\varepsilon\), we obtain the stronger convergence
\begin{align} \label{strongfrac}
    \tilde c_i^\varepsilon\to c_i
\qquad\text{strongly in }
L^2(0,T;W^{r,2}(\Omega)).
\end{align} 
Employing the second scaled trace inequality \eqref{scaledfractrace}, we have
\[\limsup_{\varepsilon \to 0} \varepsilon \|c_i^\varepsilon-c_i\|_{L^2((0,T)\times\Gamma^\varepsilon)}^2
\le
\limsup_{\varepsilon \to 0} C\|c_i^\varepsilon-c_i\|_{L^2(W^{r,2}(\Omega^\varepsilon))}^2 \le
\lim_{\varepsilon \to 0} C\|\tilde{c}_i^\varepsilon-c_i\|_{L^2(W^{r,2}(\Omega))}^2 = 0,
 \]
 hence the convergence \eqref{sconv:trc} holds true by the characterization provided in Definition \ref{def:strongtwoscalesurface}.
\end{proof}

\subsection{The macroscale problem}
We can now define the effective macroscopic problem. Find 
\(
(\vec{c}, \vec{r}) = \bigl( (c_i)_{i=1}^3, (r_j)_{j=0}^3 \bigr)
\) with
\begin{equation}\label{macroest}
    \begin{gathered}
        c_i \in L^2(0,T; W^{1,2}(\Omega)), \quad \dert c_i \in L^2(0,T; W^{-1,2}(\Omega)),\\
 r_j \in L^\infty\bigl((0,T)\times\Omega\times\Gamma_C\bigr), \quad 
\dert r_j \in L^2\bigl((0,T)\times\Omega\times\Gamma_C\bigr),
    \end{gathered}
\end{equation}
$c_i \geq 0$ and $r_j \geq 0$ a.e., such that 
\begin{align}
\varphi\partial_t c_i 
- \nabla_x \cdot \bigl(D^{\mathrm{eff}}_i \nabla_x c_i\bigr) 
&= \varphi\left(f_{\mathrm{deg},i}(c_i) 
+  f_{\mathrm{agg},i}(\vec{c})\right)
+ S_i(c_i, \vec{r})
&&\text{in } (0,T)\times\Omega, \label{macroc}\\
 D^{\mathrm{eff}}_i\nabla c_i \cdot \bn &= 0 &&\text{on } (0,T)\times\partial \Omega, \label{macrobdry}\\
\partial_t r_j 
&= f_{\mathrm{rec},j}(\vec{c}, \vec{r}) 
&&\text{on } (0,T)\times\Omega\times\Gamma_C, \label{macror}
\end{align}
 where $\varphi = {|Y|}/|Z|$ denotes the \emph{porosity}, defined as the volume fraction of the reference cell occupied by the ECS. The effective surface contribution $S_i(c_i, \vec{r})$ is given by
\begin{equation}\label{effsurface}
    S_i(c_i, \vec{r}) = \frac{1}{|Z|}\left(\int_{\Gamma_N} q_{\mathrm{prod},i}\, dS_y
+f_{\mathrm{bind},i}\left(c_i, \int_{\Gamma_C} \vec{r} \, dS_y\right)\right).
\end{equation}
The \emph{effective diffusion tensor} $D^{\mathrm{eff}}_i$ is determined by the microscopic diffusivity and the geometry of the microstructure.  We define the \emph{geometric effective diffusivity tensor} $D^{\mathrm{eff}}$, by
\begin{equation}\label{genefftensor}
({D}^{\mathrm{eff}})_{kl} = \frac{1}{|Z|}\int_{Y}\bigl(\delta_{kl} + \partial_{y_k} \chi^l(y)\bigr) \, dy,
\end{equation}
where $\chi^l$ is the solution to the cell problem
\begin{equation}\label{cellproblem}
\begin{aligned}
\Delta_y \chi^l &= 0 && \text{in } Y, \\
(\nabla_y \chi^l + {\pmb e_l}) \cdot \bn_y &= 0 && \text{on } \Gamma_C \cup \Gamma_N, \\
\int_{Y} \chi^l \, dy &= 0,
&&\text{and }\chi^l \text{ is } Z\text{-periodic}.
\end{aligned}
\end{equation}
The effective diffusion tensor of species $i$ is then obtained by scaling $D^{\mathrm{eff}}$ with the microscopic diffusion coefficient,
\begin{equation}\label{efftensor}
D^{\mathrm{eff}}_i = D_i\, D^{\mathrm{eff}}.
\end{equation}

We conclude this section with the main homogenization result. Throughout the proof, we use the standard notation \(\mathcal{D}(U)\) for the space of smooth functions with compact support in the open set \(U\). In particular, \(\mathcal{D}(\Omega;\Sper(Y))\) consists of functions that are compactly supported with respect to the macroscopic variable \(x\), \(Y\)-periodic with respect to the microscopic variable \(y\), and smooth in both variables.

\begin{thm}[Homogenization]\label{thm:homog}
Under Assumptions~\ref{asagg}--\ref{asinit}, the sequence 
$(\vec{c}_\varepsilon, \vec{r}_\varepsilon)$ of solutions to the microscopic problem \eqref{ISFeps}--\eqref{receptoreps} converges, up to a subsequence, to a weak solution $(\vec{c}, \vec{r})$ of the macroscopic problem \eqref{macroc}--\eqref{cellproblem}. 
Furthermore, the effective diffusion tensor $D^{\mathrm{eff}}$ defined by \eqref{genefftensor} is symmetric, positive definite, and bounded from above by the porosity, i.e.
\[
\bxi\cdot D^{\mathrm{eff}}\bxi \le \frac{|Y|}{|Z|}\,|\bxi|^2 \quad \text{for all } \bxi\in\rr^3.
\]
\end{thm}

\begin{proof} 
We take
\[
\varphi_i^\varepsilon(t,x) := \varphi_i(t,x) 
+ \varepsilon \varphi_i^1\!\left(t, x, \tfrac{x}{\varepsilon}\right)
\]
as a test function in \eqref{wfbulk}, where $\varphi_i \in \mathcal{D}(0,T; \mathcal{D}(\Omega))$ and 
$\varphi_i^1 \in \mathcal{D}(0,T; \mathcal{D}(\Omega; \Sper(Z)))$, and we pass to the limit $\varepsilon \to 0$. 

Using the definition of the weak derivative together with \eqref{conv:c0weak},
\begin{align*} 
\lim_{\varepsilon\to 0}&\int_0^T \langle \partial_t c_i^{\varepsilon}, \varphi_i^{\varepsilon} \rangle \dt = - \lim_{\varepsilon\to 0}\int_0^T \!\!\int_{\Omega_{\varepsilon}} c_i^{\varepsilon} \partial_t \varphi_i^{\varepsilon}\dx \dt = -\lim_{\varepsilon\to 0}\int_0^T \!\!\int_{\Omega} \hat{c}_i^{\varepsilon} \partial_t \varphi_i^{\varepsilon} \dx \dt\\&= -\int_0^T\!\! \int_{\Omega}\!\int_Z c_i \chi_Y \partial_t \varphi_i \dy \dx \dt = \int_0^T \!\! \int_{\Omega}\!\int_Y \partial_t c_i \varphi_i \dy \dx \dt \ = |Y| \int_0^T \!\! \langle \partial_t c_i, \varphi_i \rangle \dt.
\end{align*}
For the diffusion term, invoking the convergence \eqref{conv:gradc0}, we have
\begin{align*} 
\lim_{\varepsilon\to 0} \int_0^T \!\! \int_{\Omega_{\varepsilon}} D_i \nabla c_i^{\varepsilon} \cdot \nabla \varphi_i^{\varepsilon} \dx \dt = \int_0^T \!\! \int_{\Omega} \! \int_Y D_i(\nabla_x c_i + \nabla_y c_i^1) \cdot (\nabla_x \varphi_i + \nabla_y \varphi_i^1) \dy \dx \dt.
\end{align*}

We handle the binding term using the weak convergence \eqref{conv:r}, the strong convergence \eqref{sconv:trc}, and the uniform receptor bounds \eqref{estimater}: 
\begin{align*}
\lim_{\varepsilon \to 0} &\varepsilon \int_0^T \!\! \int_{\Gamma_C^{\varepsilon}} f_{\mathrm{bind}, i}(c_i^{\varepsilon}, \vec{r}_{\varepsilon}) \varphi_i^{\varepsilon} \; dS \dt = \lim_{\varepsilon \to 0} \varepsilon \int_0^T \!\! \int_{\Gamma_C^{\varepsilon}} \sum_{j=0}^3 (\beta_i^j c_i^{\varepsilon} + \gamma_i^j) r_j^{\varepsilon} \varphi_i^{\varepsilon} \; dS  \dt  \\&= \lim_{\varepsilon \to 0} \varepsilon \int_0^T \!\! \int_{\Gamma_C^{\varepsilon}} \sum_{j=0}^3 \beta_i^j (c_i^{\varepsilon} - c_i) r_j^{\varepsilon} \varphi_i^{\varepsilon} \; dS  \dt + \lim_{\varepsilon \to 0} \varepsilon \int_0^T \!\! \int_{\Gamma_C^{\varepsilon}} \sum_{j=0}^3 (\beta_i^j c_i + \gamma_i^j) r_j^{\varepsilon} \varphi_i^{\varepsilon} \; dS  \dt  \\&= \int_{\Omega}\! \int_{\Gamma_C}\sum_{j=0}^3 (\beta_i^j c_i + \gamma_i^j) r_j \varphi_i  \; dS_y  \dx \dt = \int_0^T \!\! \int_{\Omega}\! \int_{\Gamma_C} f_{\mathrm{bind},i}(c_i, \vec{r}) \; dS_y \varphi_i \dx \dt,
 \end{align*}
 where we have used 
\[
\varepsilon \left|\int_0^T \!\! \int_{\Gamma_C^{\varepsilon}} \sum_{j=0}^3 \beta_i^j (c_i^{\varepsilon} - c_i) r_j^{\varepsilon} \varphi_i^{\varepsilon} \; dS  \dt \right| 
\leq C \int_0^T \varepsilon\,\|c_i^{\varepsilon} - c_i\|_{L^2(\Gamma_C^{\varepsilon})}  \|\varphi_i^{\varepsilon}\|_{L^2(\Gamma_C^{\varepsilon})}  \|\vec{r}_\varepsilon\|_{L^\infty(\Gamma_C^{\varepsilon})}\dt \to 0
\]

The remaining bulk terms converge by \eqref{conv:c0weak} and \eqref{sconv:c}. 
Since the convergence of the production term is trivial, we obtain the limiting system
\begin{equation}\label{macromodelwf}
    \begin{aligned}
        \int_0^T \!\! &\left(|Y|\langle \partial_t c_i, \varphi_i \rangle  + \int_{\Omega} \! \int_Y D_i(\nabla_x c_i + \nabla_y c_i^1) \cdot (\nabla_x \varphi_i + \nabla_y \varphi_i^1) \dy \dx\right) \dt \\&=  \int_0^T \!\! \int_{\Omega} \! \left(\int_{\Gamma_N} q_{\mathrm{prod},i}(t,x,y) \; dS_y   +   \int_{\Gamma_C} f_{\mathrm{bind},i}(c_i, \vec{r}) \, dS_y +  |Y| \left(f_{\mathrm{deg},i}(c_i) 
+  f_{\mathrm{agg},i}(\vec{c})\right)\right)  \varphi_i \dx \dt.
    \end{aligned}
\end{equation}

To determine the unknown function $c_i^1$, we set $\varphi_i=0$ in \eqref{macromodelwf} and obtain \[\int_0^T \!\!\int_{\Omega} \! \int_Y (\nabla_x c_i(t,x) + \nabla_y c_i^1(t,x,y) )\cdot \nabla_y \varphi_i^1(t,x,y) \dy \dx \dt = 0\] 
for any $\varphi_i^1$. Hence, 
\[c_i^1(t,x,y) = \sum_{l=1}^3\partial_{x_l}c_i(t,x) \chi^l(y),\]
where $\chi^l$ solves the cell problem \eqref{cellproblem}, see \cite{Marciniak2008, RaduDipl} for more details. 

For the surface equation, we take  \[\psi^\varepsilon_j(t,x):=\psi^1_j(t,x,\tfrac{x}{\varepsilon})\] as a test function in \eqref{wfsurf}, where
$\psi^1_j \in \mathcal{D}(0,T; \mathcal{D}(\Omega; C^\infty(\Gamma_C)))$ and is $Z$-periodically extended over $\rr^3$. The convergence of the time-derivative term is secured by \eqref{conv:dtr}. The right-hand side converges by the convergences \eqref{conv:r}, \eqref{sconv:trc} and the uniform bound \eqref{estimater}, which can be shown similarly as for the binding function. We thus have 
\begin{align}\label{macromodelwfs}
   \int_0^T \!\! \int_{\Omega}\! \int_{\Gamma_C} \dert r_j \psi^1_j\; dS_y  \dx \dt  = \int_0^T \!\! \int_{\Omega}\! \int_{\Gamma_C} f_{\mathrm{rec},j}(\vec{c}, \vec{r}) \psi^1_j\; dS_y  \dx \dt.
\end{align}

Finally, setting $\varphi_i^1 = 0$ and localizing \eqref{macromodelwf} and \eqref{macromodelwfs} using 
the Fundamental Lemma of the Calculus of Variations, we obtain the desired system 
\eqref{macroc}--\eqref{macror}. Since $c_i \geq 0$ and $r_j \geq 0$ a.e.\ by Lemma~\ref{lem:twoscale}, the pair 
$(\vec{c}, \vec{r})$ obtained above is a weak solution of the macroscopic problem 
\eqref{macroc}--\eqref{cellproblem}, which proves the theorem.

Lastly, let us show that $D^{\mathrm{eff}}$ defined as in \eqref{genefftensor} is symmetric, positive definite, and bounded by porosity $\varphi=|Y|/|Z|$. By definition,
\begin{equation}\label{genefftensorZ}
|Z|(D^{\mathrm{eff}})_{kl} = \int_{Y}\delta_{kl} + \partial_{y_k} \chi^l \, dy = \int_{Y} {\pmb e_k} \cdot ( {\pmb e_l} + \nabla_y \chi^l)\, dy 
\end{equation}
Since for any $\psi\in W_{per}^{1,2}(Y)$ it holds
\[\int_Y \nabla_y \psi\cdot {\pmb e_l} \, dy = \int_{\Gamma_C \cup \Gamma_N} \psi({\pmb e_l} \cdot \bn_y) \; dS_y, \]
we can write the weak formulation of \eqref{cellproblem} as 
\begin{equation}\label{wfcellproblem}
   \int_Y \nabla_y \chi^l \cdot \nabla_y \psi\, dy = - \int_Y  {\pmb e_l} \cdot \nabla_y \psi \, dy, \quad \forall \psi\in W_{per}^{1,2}(Y).
\end{equation}
The existence and uniqueness of $\chi^l \in W_{per}^{1,2}(Y)$  solving \eqref{wfcellproblem} follow from the Lax–Milgram theorem. Taking $\varphi:= \chi^k$ in \eqref{wfcellproblem} yields \(0 = \int_Y ( {\pmb e_l} + \nabla_y \chi^l) \cdot \nabla_y \chi^k \, dy.\)
Adding this identity to \eqref{genefftensorZ} gives
\begin{equation*}
|Z|(D^{\mathrm{eff}})_{kl}
= \int_Y ({\pmb e_k} + \nabla_y \chi^k)\cdot({\pmb e_l} + \nabla_y \chi^l)\,dy.
\end{equation*}
The symmetry is thus obvious. 

To show the positive definiteness, we set $\chi^{\bxi}:=\sum_{l=1}^3\xi_l\chi^l$ for any $\bxi = (\xi_1,\xi_2,  \xi_3) \in \rr^3.$ By the linearity of the cell problem, $\chi^{\bxi}$ solves \eqref{cellproblem} with ${\pmb e_l}$ replaced by $\bxi$.
We have 
\begin{equation}\label{positivesemidef}
    \bxi \cdot |Z|D^{\mathrm{eff}} \bxi = \sum_{k,l=1}^3\xi_k\xi_l |Z|(D^{\mathrm{eff}})_{kl} = \int_Y \sum_{k=1}^3\xi_k({\pmb e_k} + \nabla_y \chi^k)\cdot \sum_{l=1}^3\xi_l({\pmb e_l} + \nabla_y \chi^l)\,dy = \int_Y |\bxi + \nabla_y \chi^{\bxi}|^2\,dy \geq 0. 
\end{equation}
Let us show that the inequality is strict for any nontrivial $\bxi$. Assume $\bxi \cdot D^{\mathrm{eff}} \bxi=0$ for $\bxi \neq {\pmb 0}$. It follows from \eqref{positivesemidef} that \[\nabla_y \chi^{\bxi} = -\bxi  \quad \text{ for a.e. }y \in Y,\]
which by the connectedness of $Y$ implies $\chi^{\bxi}(y)= -\bxi\cdot y +c$, but that contradicts the $Z$-periodicity of $\chi^{\bxi}$. 

Testing \eqref{wfcellproblem} (with $\bxi$ in place of ${\pmb e_l}$) by $\chi^{\bxi}$ yields
$\int_Y |\nabla_y\chi^{\bxi}|^2\,dy=-\int_Y \bxi\cdot\nabla_y\chi^{\bxi}\,dy$. Expanding the square in
\eqref{positivesemidef} and inserting this identity, we obtain
\[
\bxi\cdot D^{\mathrm{eff}}\bxi=\frac{|Y|}{|Z|}|\bxi|^2-\frac1{|Z|}\int_Y|\nabla_y\chi^{\bxi}|^2\,dy
\le \frac{|Y|}{|Z|}|\bxi|^2.
\]
\end{proof}

\begin{thm}\label{macrowp}
    Under Assumptions \ref{asagg}--\ref{asinit}, there exists a unique nonnegative weak solution $(\vec{c}, \vec{r})$ to the problem \eqref{macroest}--\eqref{cellproblem}. 
\end{thm}
\begin{proof}
A nonnegative solution exists by Theorem~\ref{thm:homog}. It remains to prove uniqueness.

We first show $c_i\in L^\infty(Q)$ via a Moser--Alikakos iteration \cite{Alikakos1979}, as in Lemma~\ref{boundedws}. The argument is in fact simpler here, since the surface term $S_i(c_i,\vec r)$ in \eqref{macroc} is already
homogenized and affine in $c_i$, so it enters the weak formulation as an ordinary bulk term rather than a boundary integral on $\Gamma_N^\varepsilon$ or $\Gamma_C^\varepsilon$.
Consequently, the scaled trace inequality used in Lemma~\ref{boundedws} is not needed, and testing \eqref{macroc} with $c_i^{2p-1}$ and iterating as before
yields $c_i \in L^\infty(Q)$.

Uniqueness now follows as in Lemma~\ref{uniquews}. Testing the system for two solutions $(\vec c,\vec r)$, $(\vec d,\vec s)$ with their differences $(u_i,v_j) := (c_i-d_i,\,r_j-s_j)$ and estimating the reaction, binding, and aggregation
terms via Assumptions~\ref{asagg}, \ref{asrec}, \ref{asbind} and the bound $c_i,d_i\in L^\infty(Q)$ gives
\[
\frac{d}{dt}\Bigl(\|\vec u\|_{L^2(\Omega)}^2 + \|\vec v\|_{L^2(\Omega\times\Gamma_C)}^2\Bigr)
+ \|\vec u\|_{W^{1,2}(\Omega)}^2
\le C(t)\Bigl(\|\vec u\|_{L^2(\Omega)}^2 + \|\vec v\|_{L^2(\Omega\times\Gamma_C)}^2\Bigr),
\]
with $C\in L^1(0,T)$. Here, unlike in Lemma~\ref{uniquews}, no
$\varepsilon$-weight or scaled trace estimate is needed since $\Omega$ is not perforated. Since $u_i(0)=0=v_j(0)$, Gronwall's lemma gives $\vec u = 0 = \vec v$, proving uniqueness.
\end{proof}

\section{Numerical experiments: Diffusion tensor's dependence on the cell geometry}\label{sec:num}
The homogenization result shows that the microscopic geometry of the tissue enters the macroscopic model through the effective diffusion tensor. In this section, we illustrate this dependence numerically for several representative cell geometries. The aim is to investigate how geometric features such as alignment and curvature affect the resulting macroscopic transport. The numerical results were implemented using the FEniCS finite element software \cite{AlnaesEtal2015}. The results reported here were robust with respect to the choice of numerical parameters (polynomial degree and mesh resolution). 

\subsection{Directional transport in two dimensions}
To examine how the shape of the capillary compartment affects transport, we compare two two-dimensional microstructures, as shown in Figure~\ref{fig:2d-geometries}. The first geometry serves as the reference, featuring a square capillary compartment, whereas the second geometry has the same volume fraction, $\varphi = 0.291$, but with the capillary compartment elongated in the $x_2$-direction. By keeping the volume fraction fixed, we isolate the effect of compartment shape on the effective diffusion tensor.
\begin{figure}[h]
    \centering
    \begin{subfigure}[b]{0.3\textwidth}
        \centering
        \begin{tikzpicture}[scale=0.8]
    \draw (0,0) rectangle (5,4);
    \fill[blue!30] (0.1, 0.1) rectangle (3.9, 3.9);
    \fill[orange!30] (3.975, 1.525) rectangle (4.925, 2.475);
    \end{tikzpicture}
    \caption{Square capillary}
    \end{subfigure}
    \hspace{0.5cm}
    \begin{subfigure}[b]{0.3\textwidth}
        \centering
        \begin{tikzpicture}[scale=0.8]
    \draw (0,0) rectangle (5,4);
    \fill[blue!30] (0.1, 0.1) rectangle (3.9, 3.9);
    \fill[orange!30] (4.1+0.25,0.1) rectangle (4.3375+0.25, 3.9);
    \end{tikzpicture}
    \caption{Elongated capillary}
    \end{subfigure}
\caption{{ \bf Two-dimensional microscopic geometries with volume fraction $\varphi = 0.291 $.} Geometry (A) serves as a reference microstructure, while geometry (B) contains an elongated capillary compartment in the $x_2$-direction.
}
\label{fig:2d-geometries}
\end{figure}
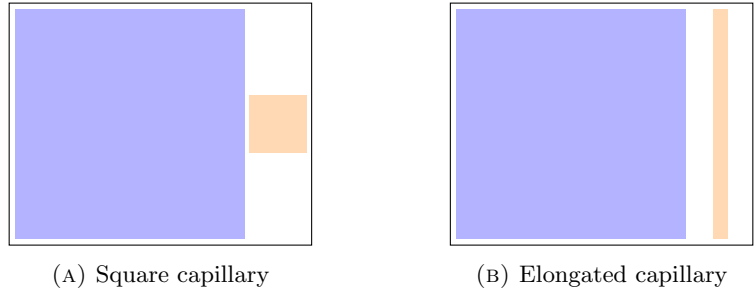

For each geometry, we solve the corresponding cell problems \eqref{cellproblem} and compute the effective diffusion tensor $D^{\mathrm{eff}}$, where each component is given by \eqref{genefftensor}. The resulting cell-problem solutions $\chi_1$ and $\chi_2$ for both geometries are shown in Figure~\ref{fig:cell-sols-grid}. 
For the reference geometry, the homogenized diffusion tensor is diagonal,
\[
D^{\mathrm{eff}}_{\mathrm{reference}} = \begin{bmatrix}
0.060 & 0 \\
0 & 0.111 \\
\end{bmatrix},
\]
indicating a higher effective diffusivity in the $x_2$-direction even in the absence of the additional preferential pathway. 
For the geometry containing the preferential pathway, the tensor remains diagonal,
\[
D^{\mathrm{eff}}_{\mathrm{elongated}} =\begin{bmatrix}
0.055 & 0 \\
0 & 0.194 \\
\end{bmatrix}.
\]
The effective diffusivity in the $x_1$-direction changes only marginally, whereas the $x_2$-component increases by approximately \(74\%\). Thus, the introduction of the preferential pathway enhances transport in the $x_2$-direction, with little effect on the orthogonal $x_1$-direction.

These results demonstrate that an isotropic microscopic diffusivity does not necessarily lead to isotropic macroscopic transport. Instead, the alignment of microscopic pathways is inherited by the effective tensor through the cell
problems.
\begin{figure}[h!]
    \centering
    \begin{subfigure}[b]{0.36\textwidth}
        \centering
        \includegraphics[width=\textwidth]{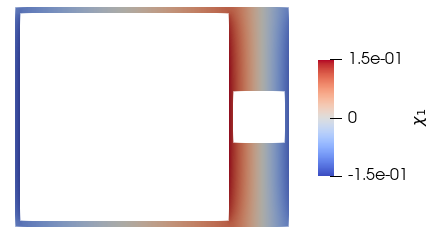}
    \end{subfigure}
    \hspace{0.5cm}
     \begin{subfigure}[b]{0.36\textwidth}
        \centering
        \includegraphics[width=\textwidth]{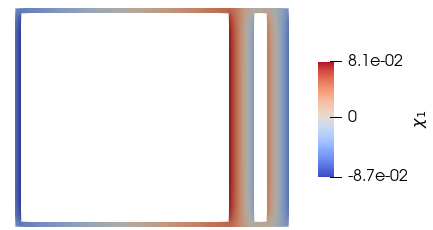}
    \end{subfigure}\\
    \vspace{0.5cm}
    \begin{subfigure}[b]{0.36\textwidth}
        \centering
        \includegraphics[width=\textwidth]{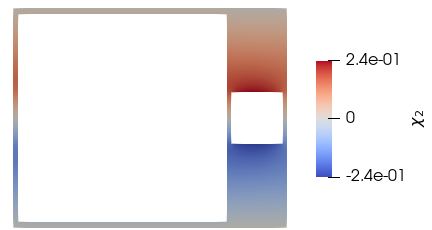}
    \end{subfigure}
    \hspace{0.5cm}
    \begin{subfigure}[b]{0.36\textwidth}
        \centering
        \includegraphics[width=\textwidth]{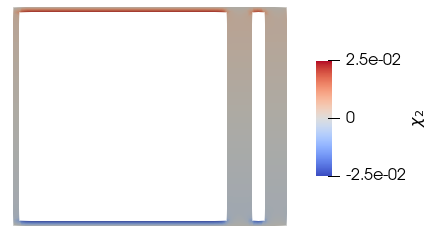}
    \end{subfigure}
    \caption{ {\bf Cell problem solutions $\chi_1$ and $\chi_2$ for both mesh geometries.} The color scale is not shared across panels: note in particular that $\chi_2$ is an order of magnitude smaller for the mesh with preferential pathways than for the reference mesh.}
    \label{fig:cell-sols-grid}
\end{figure}
 
\subsection{Effect of curvature in three dimensions}
We next consider a three-dimensional setting to investigate how the curvature of the microscopic geometry influences the effective diffusion tensor. We consider four microstructures with the same volume fraction $\varphi = 0.605$ but different interface geometries: a reference geometry with flat interfaces, a geometry with convex interfaces, a geometry with concave interfaces, and a more realistic geometry combining a concave neuronal surface with a cylindrical capillary. Since the chosen volume fraction is substantially higher than physiologically realistic values for cortical tissue, the present computations are intended as a qualitative investigation of the effects of interface geometry, while a quantitative assessment of $D^{\mathrm{eff}}$ at realistic volume fractions is left for future work. In all four geometries, the neuronal surface is symmetric with respect to all three coordinate directions, while the capillary geometry is symmetric along the $x_3$-direction. Hence, the geometries (shown in Figure~\ref{fig:3d-geometries}) do not introduce an additional preferential direction. 

\begin{figure}[htbp]
    \centering
    \begin{subfigure}[b]{0.24\textwidth}
        \centering
        \includegraphics[width=\textwidth]{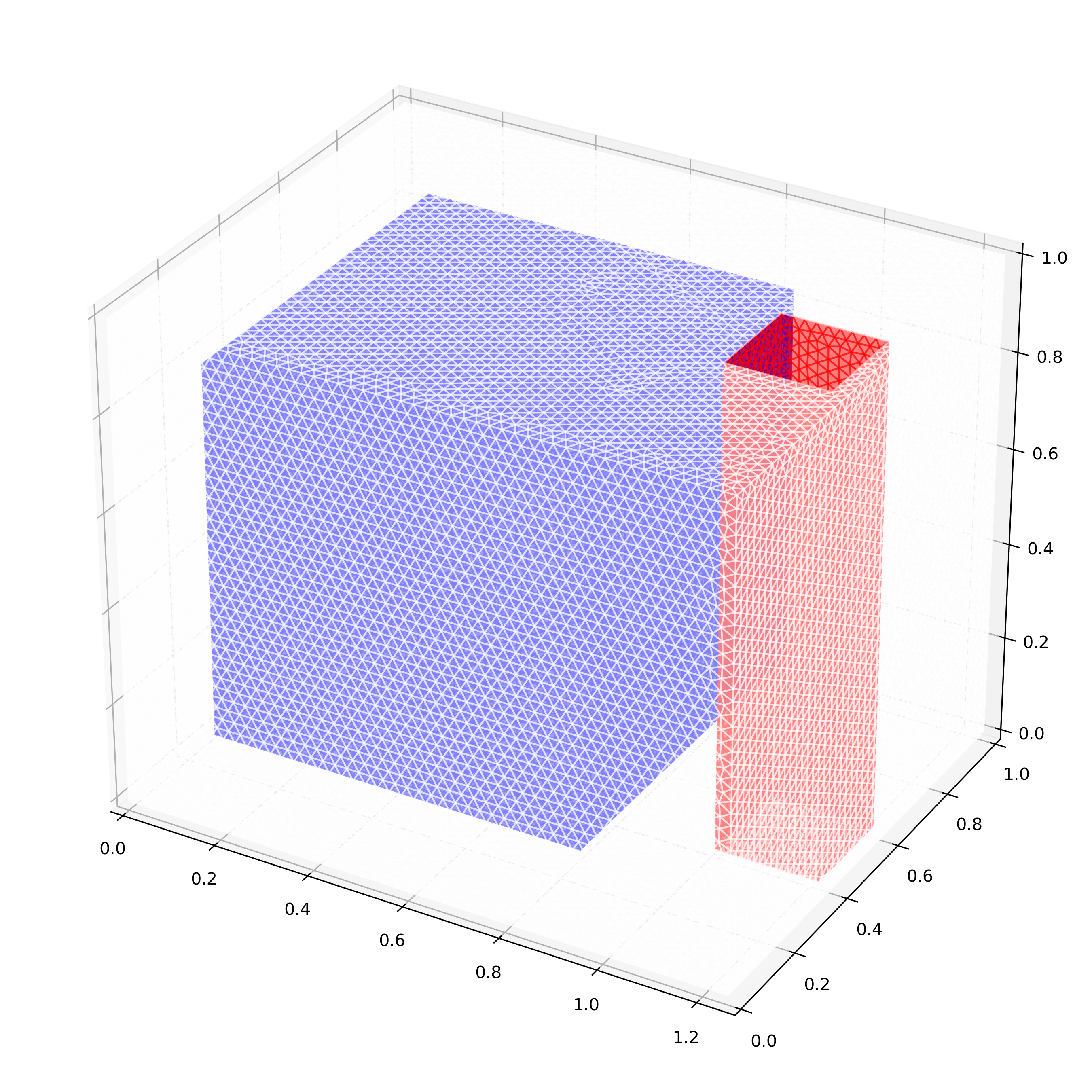}
        \caption{Flat interface}
        \label{fig:mesh-flat}
    \end{subfigure}
    \hfill
    \begin{subfigure}[b]{0.24\textwidth}
        \centering
        \includegraphics[width=\textwidth]{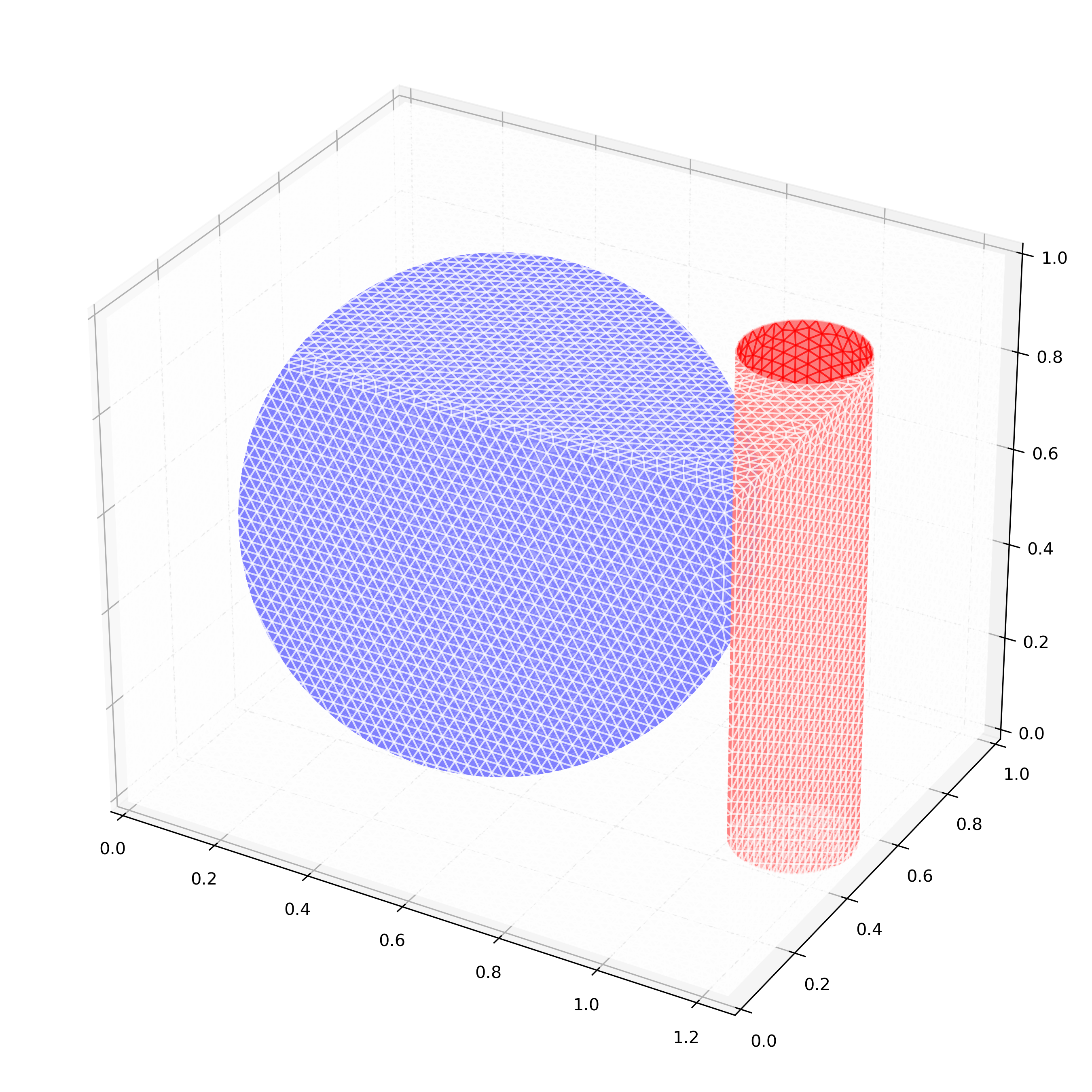}
        \caption{Convex interface}
        \label{fig:mesh-convex}
    \end{subfigure}
    \hfill
    \begin{subfigure}[b]{0.24\textwidth}
        \centering
       \includegraphics[width=\textwidth]{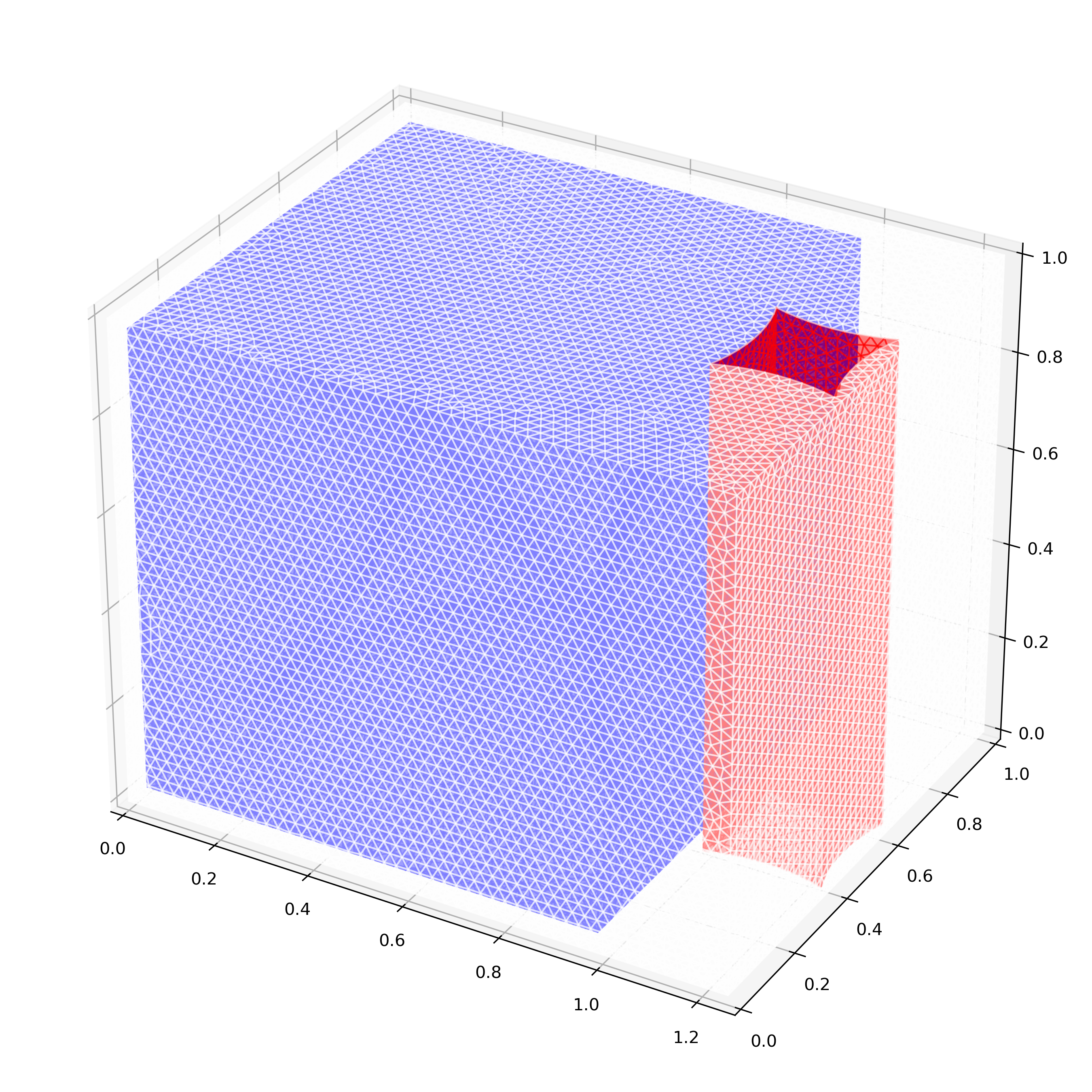}
        \caption{Concave interface}
        \label{fig:mesh-concave}
\end{subfigure}
\hfill
    \begin{subfigure}[b]{0.24\textwidth}
        \centering
       \includegraphics[width=\textwidth]{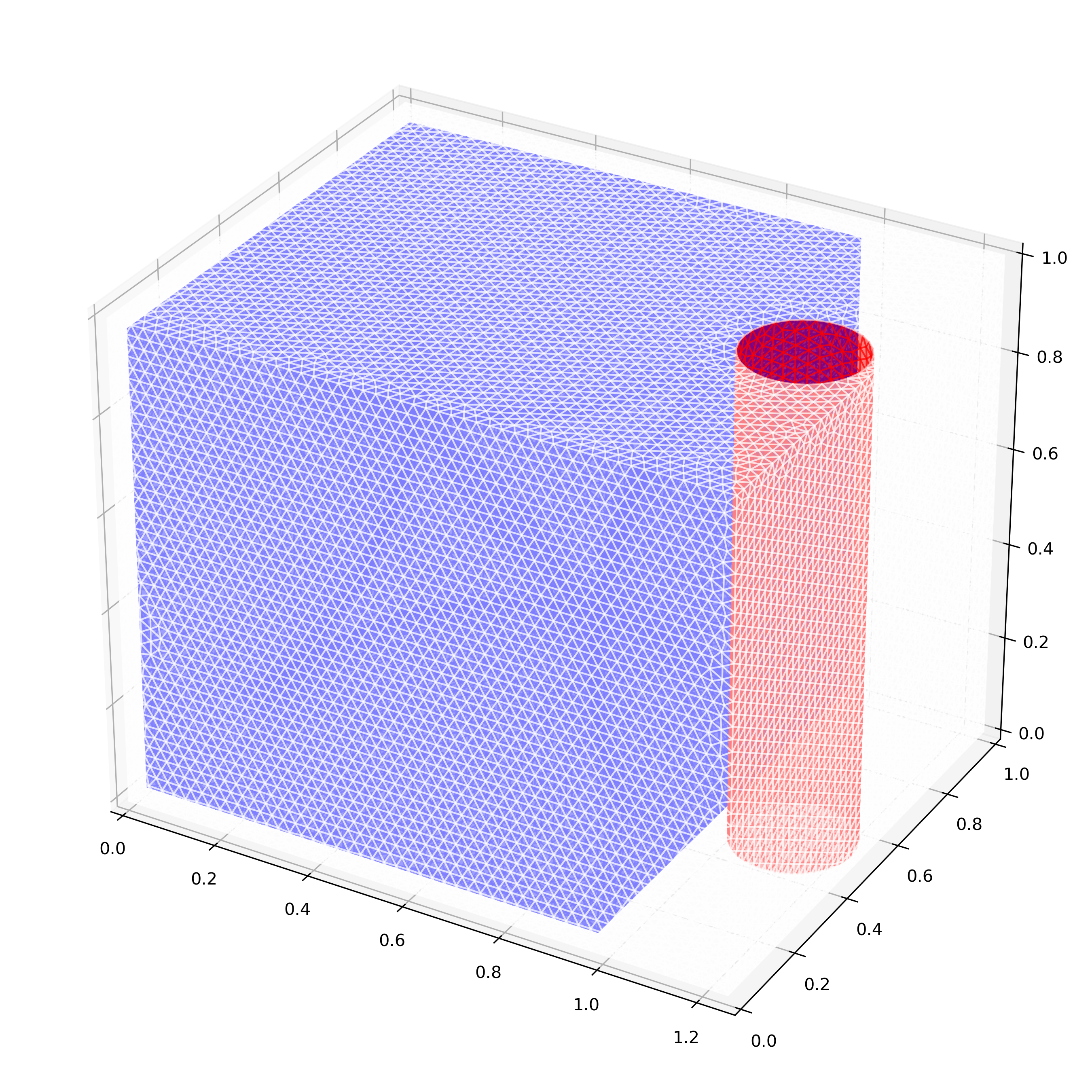}
        \caption{Mixed interface}
        \label{fig:mesh-mixed}
\end{subfigure}
\\
\begin{subfigure}[b]{0.22\textwidth}
        \centering
        \includegraphics[width=\textwidth]{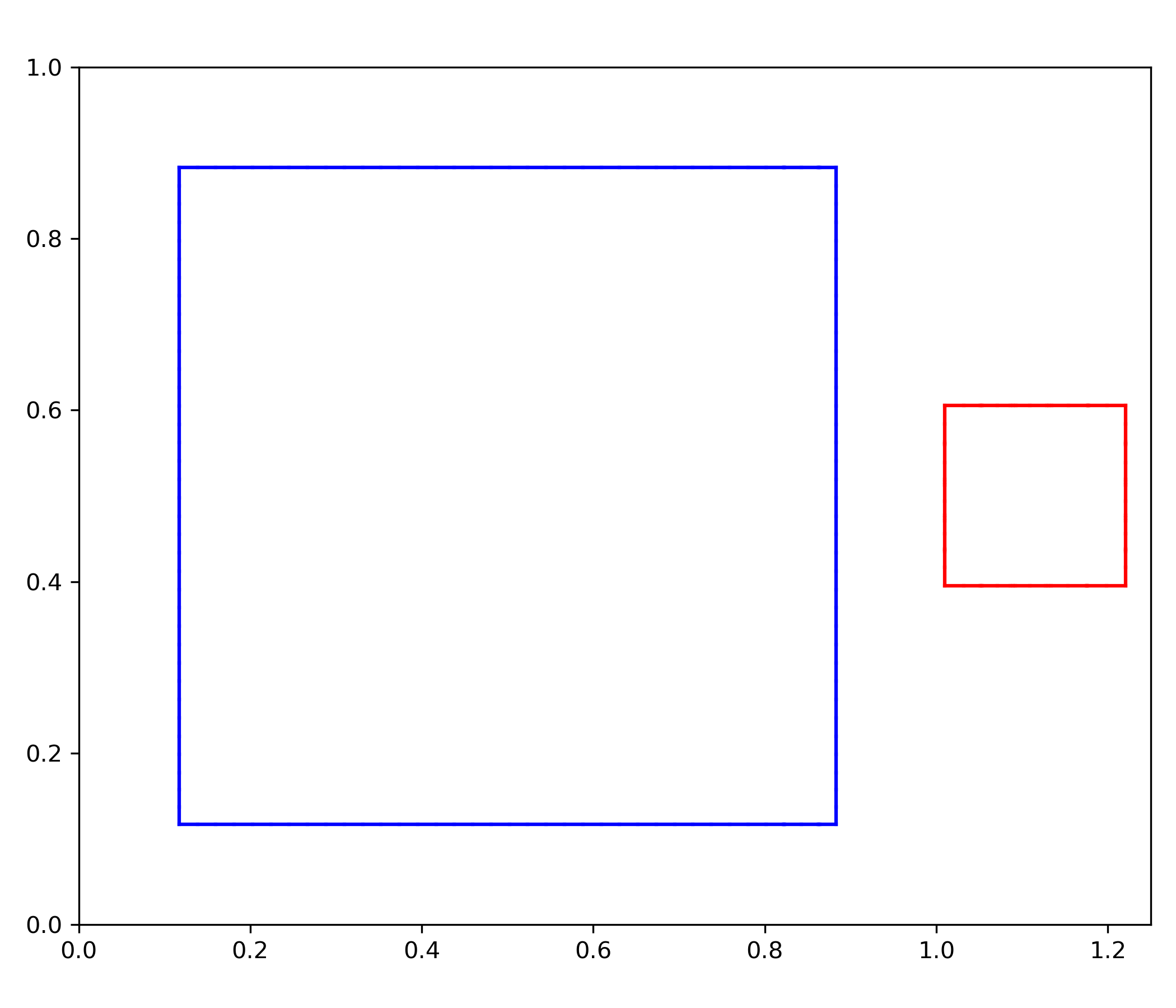}
    \end{subfigure}
    \hfill
    \begin{subfigure}[b]{0.22\textwidth}
        \centering
        \includegraphics[width=\textwidth]{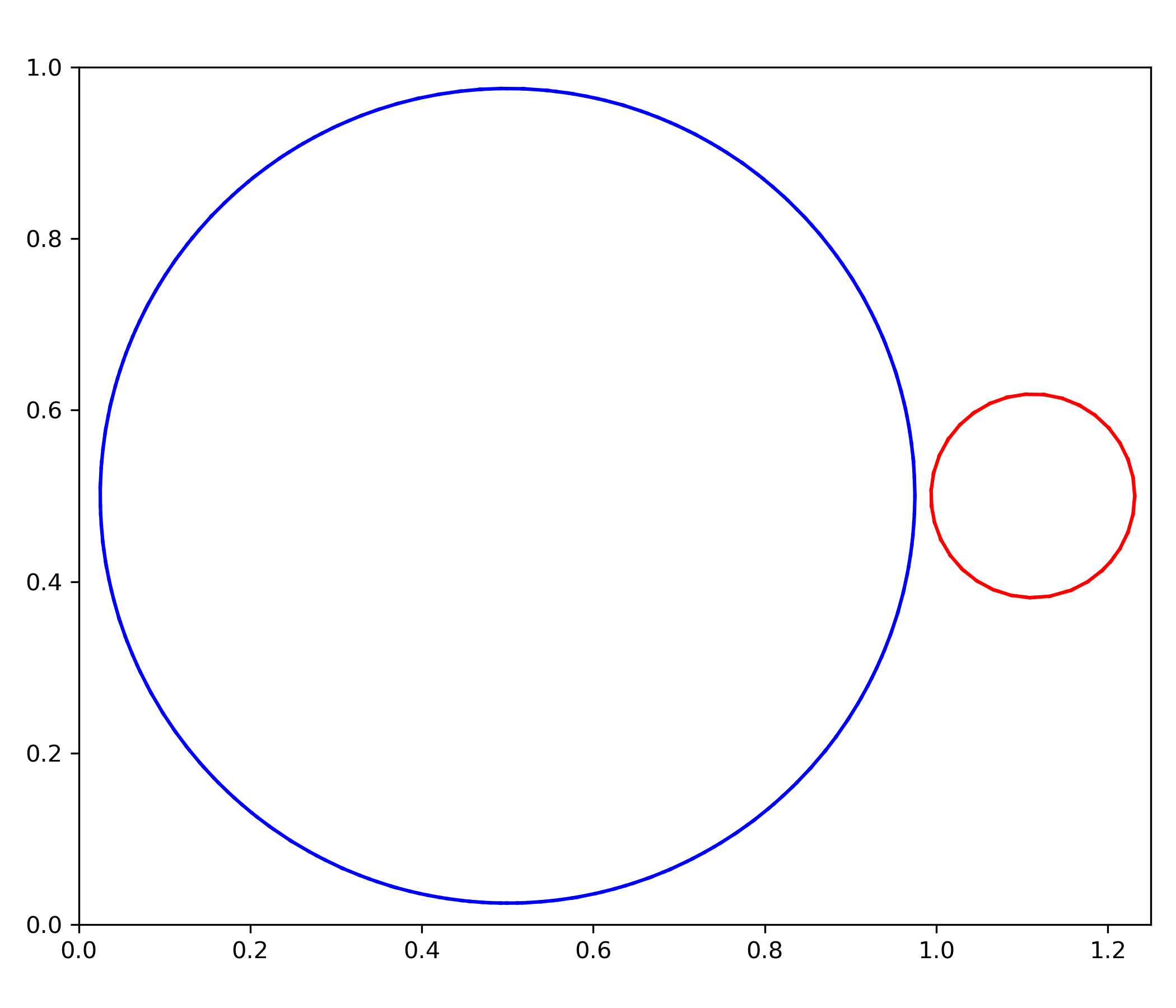}
    \end{subfigure}
    \hfill
    \begin{subfigure}[b]{0.22\textwidth}
        \centering
        \includegraphics[width=\textwidth]{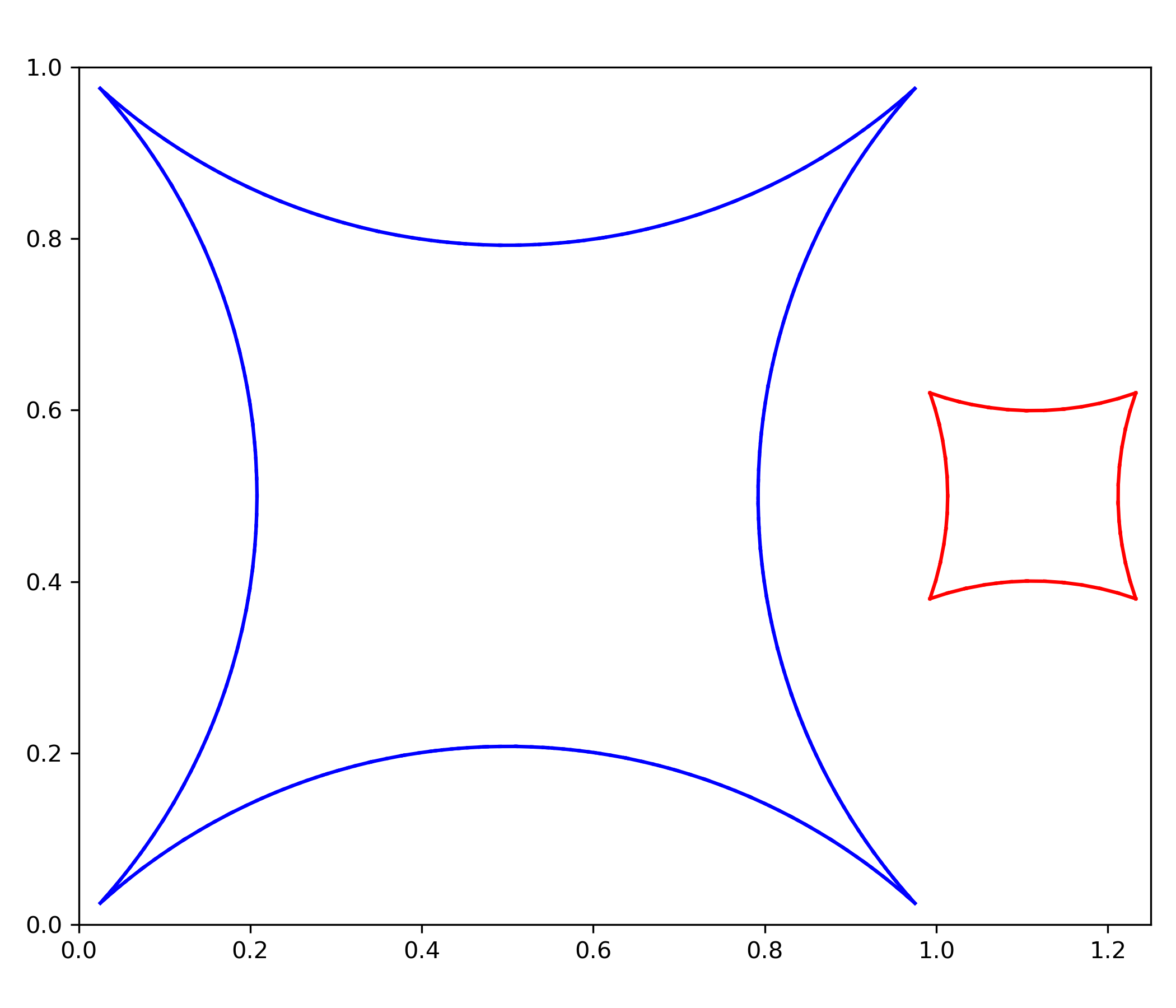}
    \end{subfigure}
    \hfill
    \begin{subfigure}[b]{0.22\textwidth}
        \centering
        \includegraphics[width=\textwidth]{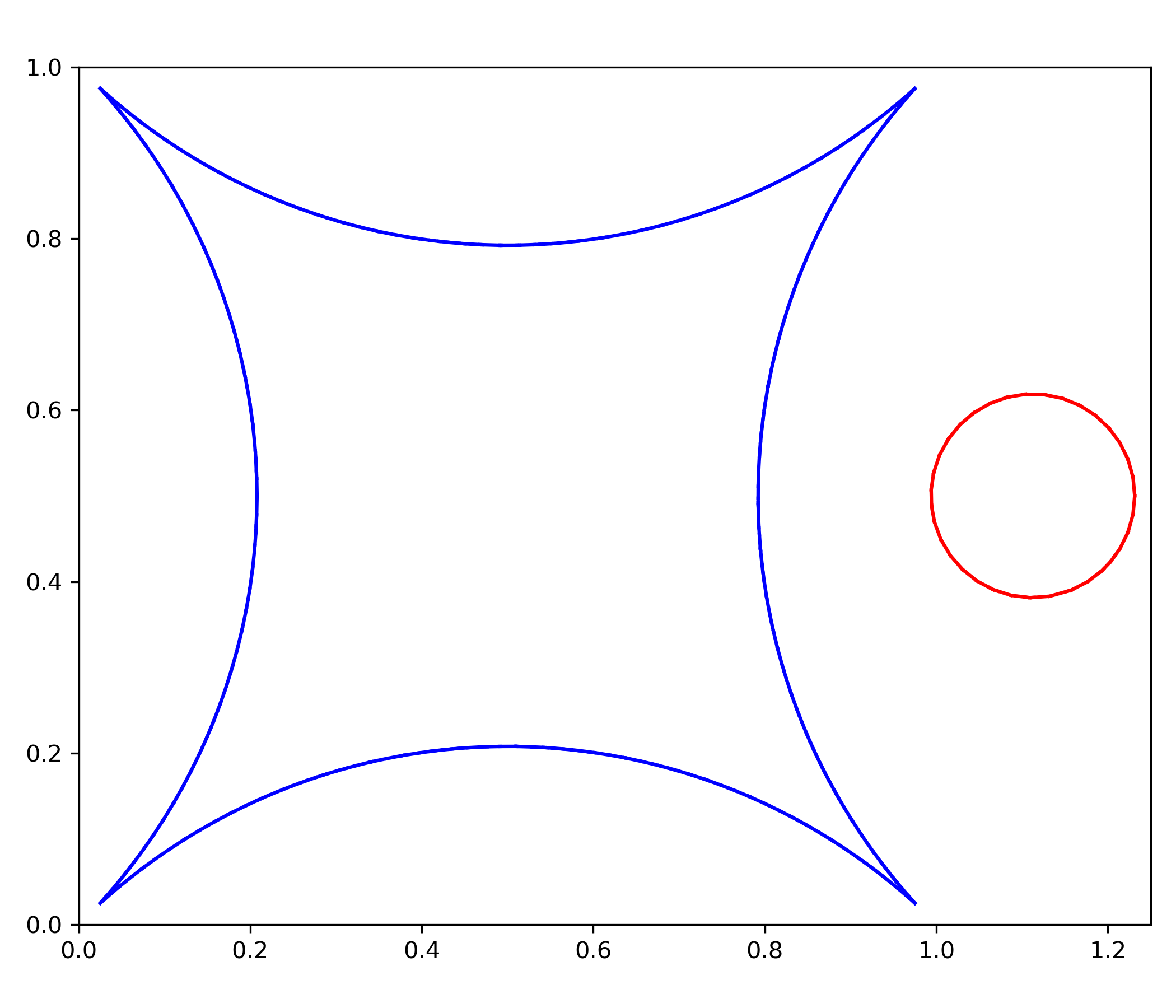}
    \end{subfigure}
    
\caption{\textbf{Three-dimensional geometries with volume fraction $\varphi = 0.605$.} The reference geometry (A) features flat interfaces, while (B) and (C) incorporate convex and concave deformations of the capillary and neuronal surfaces, respectively. Geometry (D) combines a concave neuronal surface with a cylindrical capillary. The upper panels show the full three-dimensional geometries, while the lower panels show their intersections with the plane $z=0.5$.}
\label{fig:3d-geometries}
\end{figure}

For each of the four geometries, we solve the corresponding three-dimensional cell problems and compute the effective diffusion tensor $D^{\mathrm{eff}}$ according to \eqref{genefftensor}. For the flat, convex and concave geometries, the resulting tensors are
$$
D^{\mathrm{eff}}_{\mathrm{flat}} =
\begin{bmatrix}
0.470 &0 &0 \\
0 & 0.471 &0 \\
0 &0 & 0.515 \\
\end{bmatrix},
\quad
D^{\mathrm{eff}}_{\mathrm{convex}} =
\begin{bmatrix}
0.472 &0 &0 \\
0 & 0.456 &0 \\
0 &0 & 0.519 \\
\end{bmatrix},
\quad
D^{\mathrm{eff}}_{\mathrm{concave}} =
\begin{bmatrix}
0.271 & 0 & 0 \\
0 & 0.368 & 0 \\
0 & 0 & 0.388 \\
\end{bmatrix}.
$$
 The flat and convex geometries yield comparable effective diffusivities, with the largest relative change being a decrease of approximately $3\%$ in the $x_1$-direction. In contrast, the concave geometry leads to a substantial reduction in all three components. Relative to the flat geometry, the effective diffusivities in the $x_1$-, $x_2$-, and $x_3$-directions decrease by approximately $41\%$, $21\%$, and $24\%$, respectively. Thus, in this case, the concave geometry has a considerably stronger effect on effective transport than the convex geometry.

Finally, we consider the more realistic setting of Figure~\ref{fig:3d-geometries}(D), which combines the concave neuronal surface of the concave geometry with the cylindrical capillary of the convex geometry.
The resulting effective diffusion tensor is
$$
D^{\mathrm{eff}}_{\mathrm{mixed}} =
\begin{bmatrix}
0.273 & 0 & 0 \\
0 & 0.380 & 0 \\
0 & 0 & 0.388 \\
\end{bmatrix}.
$$
Compared with the concave geometry, which differs only in the capillary shape, the diffusivities change by less than $4\%$ in all directions. Compared with the convex geometry, they decrease by approximately $41\%$, $16\%$, and $25\%$. Hence, the curvature of the neuronal surface has a considerably larger influence on the effective transport than the curvature of the capillary.

These results demonstrate that the effective diffusion tensor depends not only on the volume fraction and the orientation of preferential pathways, but also on the detailed geometry of the microscopic interfaces. In particular, while the preferential pathways determine the directionality of transport, changes in interface curvature can substantially modify the magnitude of the effective diffusivities.

\newpage
\appendix
\section{Homogenization theory}\label{homgtheory}
We rely on the two-scale convergence with parameters to derive effective macroscopic equations from microscopic heterogeneities. This concept formalizes how sequences of functions with fine-scale oscillations (arising from periodic microstructures) can converge to functions that capture both macro- and micro-scale behavior. We refer the reader to \cite{Allaire1992, AllaireDamlamianHornung, Nguetseng, RaduDipl, NeussRadu} for proofs and further technical details.

\begin{df}[Two-scale convergence]\label{def:twoscale}
 A sequence $\{u_\varepsilon\}$ in $L^2(\Omega)$ is two-scale convergent to $u \in L^2(\Omega \times Z)$ if 
\begin{equation*}
    \lim_{\varepsilon \to 0} \int_\Omega u_\varepsilon(x) \varphi\left(x, \frac{x}{\varepsilon}\right) dx = \int_{\Omega \times Z} u(x, y) \varphi(x, y) dy dx
\end{equation*}
for any $\varphi \in L^2(\Omega; C_{\text{per}}(Z))$.
\end{df}

\begin{thm}[Two-scale convergence on oscillating surfaces]\label{def:manifoldtwoscale} Let $\{v_\varepsilon\} \subset L^2(\Gamma^\varepsilon)$ satisfy $\varepsilon \|v_\varepsilon\|^2_{L^2(\Gamma^\varepsilon)} \le C$, then there exists a two-scale limit $v \in L^2(\Omega; L^2(\Gamma))$ such that, up to a subsequence, $v_\varepsilon$ two-scale converge to $v$ in the sense that
\begin{equation*}
    \lim_{\varepsilon \to 0} \varepsilon \int_{\Gamma^\varepsilon} v_\varepsilon(x) \varphi\left(x, \frac{x}{\varepsilon}\right) dS = \int_{\Omega \times \Gamma} v(x, y) \varphi(x, y) dS_y dx
\end{equation*}
for any $\varphi \in C^0(\Omega; C_{\text{per}}(Z))$.
\end{thm}

\begin{df}[Strong two-scale convergence on oscillating surfaces]\label{def:strongtwoscalesurface}
A sequence $\{v_\varepsilon\} \subset L^2(\Gamma^\varepsilon)$ satisfying 
$\varepsilon \|v_\varepsilon\|^2_{L^2(\Gamma^\varepsilon)} \le C$ is said to converge 
strongly two-scale to $v \in L^2(\Omega \times \Gamma)$ if it two-scale 
converges to $v$ in the sense of Theorem~\ref{def:manifoldtwoscale} and, additionally,
\begin{equation}\label{strongtwoscale}
    \lim_{\varepsilon \to 0} \varepsilon \int_{\Gamma^\varepsilon} 
    \left| v_\varepsilon(x) \right|^2 dS 
    = \int_{\Omega \times \Gamma} \left| v(x,y) \right|^2 dS_y\, dx.
\end{equation}
Equivalently, this means
\begin{equation*}
    \lim_{\varepsilon \to 0}\, \varepsilon \int_{\Gamma^\varepsilon} 
    \left| v_\varepsilon(x) - v\!\left(x, \tfrac{x}{\varepsilon}\right) \right|^2 dS = 0.
\end{equation*}
\end{df}

\begin{thm}[Two-scale compactness]\label{thm:cpcttwoscale} The following holds:
\begin{enumerate}
    \item[i.] If $\{u_\varepsilon\}$ is bounded in $L^2(\Omega)$, there exists a subsequence (not relabeled) such that $u_\varepsilon \rightharpoonup u$ is two-scale as $\varepsilon \to 0$ to some function $u \in L^2(\Omega \times Z)$.
    \item[ii.] If $u_\varepsilon \rightharpoonup u$ weakly in $W^{1,2}(\Omega)$ then $u_\varepsilon \rightharpoonup u$ and $\nabla u_\varepsilon \rightharpoonup \nabla u + \nabla_y u_1$ two-scale, where $u_1 \in L^2(\Omega; W^{1,2}_\text{per}(Z)/\rr)$.
    \item[iii.] If $\|u_\varepsilon\|_{W^{1,2}(\Omega^\varepsilon)} \le C$ and $\hat{u}_\varepsilon$ and $\hat{\nabla u_\varepsilon}$ are extensions by zero from $\Omega^\varepsilon_e$ into $\Omega$ of $u_\varepsilon$ and $\nabla u_\varepsilon$ respectively, then, up to a subsequence, $\hat{u}_\varepsilon$ and $\hat{\nabla u_\varepsilon}$ converge two-scale to $u \chi_Y$ and $[\nabla u + \nabla_y u_1] \chi$ respectively, where $\chi_Y = \chi_Y(y)$ is the characteristic function of $Y$, $u \in W^{1,2}(\Omega)$ and $u_1 \in L^2(\Omega; W^{1,2}_{\text{per}}(Y)/\mathbb{R})$.
\end{enumerate}
\end{thm}

\bibliographystyle{plain}
\bibliography{bibliography}

\end{document}